\documentclass[11pt, a4paper]{amsart}
\usepackage{amsmath, amssymb, amsthm, stmaryrd, latexsym}
\usepackage{enumerate, color, hyperref} 
\usepackage[all]{xy}
\usepackage{tabularx}
\usepackage[latin1]{inputenc}

\newtheorem{thm}{Theorem}[section]
\newtheorem{prop}[thm]{Proposition}
\newtheorem{lemma}[thm]{Lemma}
\newtheorem{cor}[thm]{Corollary}
\theoremstyle{definition}

\newtheorem{rem}[thm]{Remark}

\newcommand{\C}{\mathbb{C}}
\newcommand{\F}{\mathbb{F}}

\newcommand{\Q}{\mathbb{Q}}

\newcommand{\T}{\mathbb{T}}
\newcommand{\Z}{\mathbb{Z}}

\newcommand{\Fbar}{\overline{\F}}
\newcommand{\Qbar}{\overline{\Q}}
\newcommand{\Zbar}{\overline{\Z}}

\newcommand{\cA}{{\mathcal{A}}}
\newcommand{\cC}{{\mathcal{C}}}

\newcommand{\cM}{\mathcal{M}}

\newcommand{\cN}{\mathcal{N}}
\newcommand{\cO}{\mathcal{O}}
\newcommand{\cP}{\mathcal{P}}
\newcommand{\cI}{\mathcal{I}}
\newcommand{\cX}{\mathcal{X}}
\newcommand{\cY}{\mathcal{Y}}

\newcommand{\fb}{\mathfrak{b}}
\newcommand{\fc}{\mathfrak{c}}
\newcommand{\fd}{\mathfrak{d}}
\newcommand{\fm}{\mathfrak{m}}
\newcommand{\fn}{\mathfrak{n}}
\newcommand{\fo}{\mathfrak{o}}
\newcommand{\fp}{\mathfrak{p}}
\newcommand{\fq}{\mathfrak{q}}
\newcommand{\fr}{\mathfrak{r}}

\newcommand{\Ra}{\mathrm{Ra}}

\DeclareMathOperator{\End}{End}
\DeclareMathOperator{\Frob}{Frob}
\DeclareMathOperator{\Gal}{Gal}
\DeclareMathOperator{\GL}{GL}
\DeclareMathOperator{\PGL}{PGL}

\DeclareMathOperator{\id}{Id}
\DeclareMathOperator{\Spec}{Spec}
\DeclareMathOperator{\Spf}{Spf}

\DeclareMathOperator{\rH}{H}

\DeclareMathOperator{\Nm}{N_{F/\Q}}
\DeclareMathOperator{\Hom}{Hom}
\DeclareMathOperator{\Tr}{Tr}

\title[Unramifiedness of Galois representations mod~$p$]{
Unramifiedness of Galois representations attached to Hilbert modular forms mod~$p$ of weight $1$}
\author{Mladen Dimitrov \and Gabor Wiese}

\begin{document}
\maketitle

\begin{abstract}
The main result of this article states that the Galois representation attached to
a Hilbert modular eigenform  defined over~$\Fbar_p$ of parallel weight $1$ and level prime to $p$ is unramified above~$p$. 
This includes the important case of eigenforms that do not lift to Hilbert
modular forms in characteristic~$0$ of parallel weight $1$.
The proof is based on the observation that parallel weight $1$ forms 
in characteristic~$p$ embed into the ordinary part of parallel weight~$p$ forms in two different ways per prime
dividing~$p$, namely via `partial'  Frobenius operators.

\noindent
MSC: 11F80 (primary); 11F41, 11F33

\noindent
Keywords: Hilbert modular forms modulo~$p$, weight one, Galois representations
\end{abstract}

\section{Introduction}

Let $p$ be any prime number. 
Under Serre's Modularity Conjecture~\cite{serre}, proved by  Khare and Wintenberger~\cite{khare-wintenberger}, Hecke eigenforms for $\GL_2(\Q)$ over~$\Fbar_p$ correspond exactly to two-dimensional, odd, semi-simple  continuous representations
of the absolute Galois group of~$\Q$ over~$\Fbar_p$.  Using `minimal weights', as suggested by Serre and taken up by Edixhoven~\cite{edixhoven},
eigenforms of weight $1$ and level prime to~$p$ give rise to Galois representations that have the special property 
of being unramified at~$p$, and the converse  also holds at least when $p>2$ (see \cite{coleman-voloch},\cite{gross},\cite{wiese}).

In analogy  it is conjectured that two-dimensional, totally odd, semi-simple continuous representations over~$\Fbar_p$  of the absolute Galois group of a totally real number field~$F$ which are  unramified at a prime $\fp$ of $F$ dividing~$p$ correspond exactly to Hecke eigenforms for $\GL_2(F)$ over~$\Fbar_p$ of level prime to~$\fp$ and partial weight $1$ at $\fp$.

Our main theorem establishes one direction of this correspondence for parallel weight $1$ 
in full generality. Let  $G_F$ be the absolute Galois group of $F$.  
Fix for each prime $\fq$ of $F$  a decomposition group $D_{\fq}$ of $G_F$ and an arithmetic Frobenius  $\Frob_\fq$.

\begin{thm}\label{thm:main}
Let $f$ be a Hilbert modular form over $\Fbar_p$ of parallel weight $1$ and level~$\fn$ which is relatively prime to~$p$.
Assume that $f$ is a common eigenvector for the Hecke operators $T_\fq$ for all $\fq$ outside a finite set~$\Sigma$ of primes of~$F$ containing those dividing~$\fn$. 
Then there exists a continuous semi-simple representation
$$\rho_f: G_F \to \GL_2(\Fbar_p)$$
which is unramified outside $\fn$, and satisfies that for all primes $\fq\notin\Sigma$ the trace of 
$\rho_f(\Frob_\fq)$ equals the eigenvalue of $T_\fq$ on~$f$. 

In particular, $\rho_f$ is unramified at all primes $\fp$ dividing~$p$.
\end{thm}

This theorem has been proved independently by Emerton, Reduzzi and Xiao~\cite{ERX} under the additional assumptions that $p$ is inert in $F$ and the Hecke polynomial at $p$ has two distinct roots: the $p$-distinguished case. The converse has been proved in the work on  companion forms~\cite{gee-kassaei} of Gee and Kassaei   under the assumptions that $p$ is odd, $\rho$ is modular and $p$-distinguished,  and some other hypotheses.

We see our main result as an important step towards a Serre type modularity conjecture over~$F$ which includes as fine arithmetic information as possible (in particular, on ramification).
Serre's Modularity Conjecture has been formulated for  $\GL_2$ over~$F$ by Buzzard, Diamond and Jarvis~\cite{BDJ}  {\em et al.}, and the weight part of this conjecture has already been established (see~\cite{GLS}). However, due to its formulation in terms of Betti cohomology, the Buzzard-Diamond-Jarvis conjecture excludes the weight $1$ forms.

Extending earlier work of Deligne and Serre~\cite{deligne-serre}, Ohta~\cite{ohta},
Rogawski and Tunnell~\cite{rogawski-tunnell} attach to any complex holomorphic parallel weight $1$ Hilbert modular newform
a two-dimensional irreducible totally odd Artin representation of~$G_F$ which is unramified outside the 
level of the form.
In a series of works by Kassaei, Pilloni, Sasaki, Stroh and Tian, building on a strategy of Buzzard and Taylor~\cite{buzzard-taylor},
it has been shown that the converse also holds.
However, since Hilbert modular forms over~$\Fbar_p$ do not necessarily lift to characteristic~$0$ in the same weight,
the main result of the present article cannot be deduced from those works. 

In  \cite{jasper}, Jasper Van Hirtum describes an algorithm for computing spaces of Hilbert modular forms of parallel weight $1$,
which we expect to provide explicit examples of non-liftable parallel weight $1$ Hilbert eigenforms.
In the perspective of the inverse Galois problem, this will provide examples of Galois extensions of~$F$ with group $\PGL_2(\F_p)$
the unramifiedness above~$p$ of which would follow from our main theorem, but would be very hard to check without it since one would
in general not be able to write down an explicit polynomial giving the extension.
Such fields are also expected to have a small root discriminant.

A substantial part of the paper is devoted to constructing, for  $\fp$ a prime dividing~$p$, a 
Hecke operator $T_\fp^{(k)}$ acting on the space of Hilbert modular forms over~$\F_p$  of parallel weight $k\geq 1$ and  level which is relatively prime to $p$.  Although $T_\fp^{(k)}$ is  uniquely determined  by its action on  $q$-expansions, hence {\it a fortiori} by its action on $p$-adic modular forms (over the ordinary locus), its definition in the current state of the knowledge seems to require some geometric input from the non-ordinary locus of the Hilbert modular variety. Our construction relies on a fine geometric result from \cite{ERX} stating roughly that for our purposes the non-ordinary locus  can be ignored (see \S\ref{sec:3.3.4}). A different construction of $T_\fp^{(k)}$ was recently announced by V.~Pilloni. 
 
A specific  feature of our treatment is that we compute the action of $T_\fp^{(k)}$ on adelic $q$-expansions, which we
expect  to be also useful for  questions on Hilbert modular forms over~$\Fbar_p$ beyond the ones treated in this article.

\begin{thm}[Effect of $T_\fp$ on $q$-expansions]\label{thm:Tp-q}
For any Hilbert modular form over~$\Fbar_p$ of parallel weight $k\geq 1$ and prime to~$p$ level,  and for any prime $\fp$ of $F$ dividing $p$
we have: 
 \begin{equation}\label{eq:Tp-general}
\begin{split}
a((0), T_\fp^{(k)} f) & = a((0),f)[\fp] +\Nm(\fp)^{k-1} a((0),\langle\fp\rangle f)[\fp^{-1}],   \text{ and}\\
a(\fr, T_\fp^{(k)} f) & = a(\fp \fr,f) +\Nm(\fp)^{k-1} a(\fr/\fp,\langle\fp\rangle f) \text{ for all ideals } (0)\neq\fr\subset \fo.
\end{split}
\end{equation}

Here  $a(\fr, f)$ denotes the $\fr$-th coefficient in the adelic $q$-expansion of $f$ (see \S\ref{adelic-q-exp}) and
$\langle\fp\rangle$ denotes the diamond operator (see \S\ref{diamond}). 
\end{thm}

\noindent 
Let us now highlight some  delicate points in   the proof of the main theorem. We will only work with parallel weights  so the word `parallel' will henceforth be dropped.

Our definition of Hilbert modular forms  is in the spirit of Wiles and corresponds to  Katz modular forms
which are invariant under the action of a  finite group. More precisely, as explained in \S\ref{hmf},  we use forms living on the Hilbert modular variety
classifying Hilbert-Blumenthal abelian varieties modulo the action of the totally positive units on polarisations.
Working with modular forms \`a la Wiles is necessary in order to have a good Hecke theory, {\em i.e.}, it seems to be the right setting
for working with Galois representations. Moreover,  over~$\C$  those are precisely  the
automorphic forms on $\GL_2$ studied by Shimura in~\cite{shimura}.
The technical complications arising from the fact that the relevant Hilbert modular variety is only a finite quotient of the fine moduli scheme are properly addressed.
The literature is not as complete and as uniform in terms of definitions as
one would wish; we have nonetheless tried to include as many and as precise references as possible.

Our general strategy is to go to weight~$p$ where we exploit the subtle phenomenon of `doubling' for  Hecke algebras
(failure to strong multiplicity one in characteristic~$p$; see \cite{wiese}).
Whereas the overall strategy of proof is very similar to the one in \cite{wiese}, the actual techniques used here
are quite different since many of the arguments in {\em loc.\ cit.} seem to be  specific to $F=\Q$.

A crucial insight in the non-$p$-distinguished case is that a certain ordinary modular deformation of $\rho_f$ is not ordinary over its algebra of traces (see Remark \ref{u-operator}).

\subsection*{Further directions}
We  expect that for any $f$ as in Theorem \ref{thm:main} for which $\rho_f$ is irreducible one can  prove the unramifiedness above~$p$ of the Galois representation
with coefficients in the  weight $1$ Hecke algebra $\T$ over~$\Fbar_p$ of level $\fn$ localised at
the maximal ideal corresponding to $f$. 
It is further expected that $\T$ is in fact the universal deformation Artinian $\Fbar_p$-algebra of $\rho_f$ unramified outside~$\fn$  with prescribed local properties at primes dividing~$\fn$.
This has been proved  for   $F=\Q$ and $p>2$  by Calegari and Geraghty~\cite{calegari-geraghty}, who even establish more, namely that all unramified at~$p$ minimal deformations of $\rho_f$ are modular of weight $1$.
Implementing the  method of Calegari and Geraghty for  arbitrary $F$ needs as a crucial ingredient the unramifiedness above~$p$ of the Galois representations attached to eigenclasses in  the coherent cohomology of the Hilbert modular variety in any degree (not only to those  corresponding to modular forms which 
occur in degree $0$).  This seems to require new ideas (see \cite{ERX} for a partial result in the real quadratic case).

\subsection*{Acknowledgements}

We would like to thank heartily Hwajong Yoo for his interest in this project and for many stimulating discussions. Our thanks go next to  Beno\^\i{}t Stroh, Vincent Pilloni and Liang Xiao for generously sharing their ideas on how one can define the Hecke operators in positive characteristic. Thanks are also due to the referee, to Pierre Charollois and  Brian Conrad for providing us with some useful bibliographical references, and to Fabrizio Andreatta, Frank Calegari, Payman Kassaei and Shu Sasaki for helpful discussions.

Further we thank Kai-Wen Lan for pointing out a problem in an earlier version of this manuscript, which was resolved using some recent results of Emerton, Reduzzi and Xiao, which were generously communicated to us by the latter.  

The first named author (M.D.) would like to thank the University of Luxembourg -- where most of our discussions took place -- for its hospitality. The second named author (G.W.) acknowledges support for this research by the University of Luxembourg Internal Research Project GRWTONE and also partial support by the Fonds National de la Recherche Luxembourg INTER/DFG/12/10/COMFGREP.

\section{Geometric Hilbert modular forms}\label{sec:HMF}

In this section we recall and adapt some results from 
Andreatta-Goren~\cite{andreatta-goren}, Chai~\cite{chai}, Deligne-Pappas~\cite{deligne-pappas}, 
Dimitrov \cite{dim-hmv,dimdg}, Dimitrov-Tilouine~\cite{dimtildg}, Hida \cite{hida-PAF}, Kisin-Lai~\cite{kisin-lai} and 
Rapoport~\cite{rapoport}.

Let $F$ be a totally real field of degree $d\geq 2$, ring of integers~$\fo$ and different~$\fd$.  Fix an ideal $\fn\subset \fo$ (the level). 
For any fractional ideal  $\fc$ of $F$ we denote by $\fc_+=\fc\cap F_+^\times$  the cone of its 
totally positive elements.  The set of equivalence classes of $\fc$ modulo the action of $F_+^\times$ 
can be naturally identified with the narrow class group $\cC\ell_F^+$ of $F$. 
By abuse of notation, we will also denote by $\cC\ell_F^+$ a fixed set of representatives.

\subsection{Hilbert modular varieties}\label{hmv}

Recall that a Hilbert-Blumenthal abelian variety  over a scheme $S$ is an abelian scheme $\pi:A\rightarrow S$ of relative dimension $d$ together with an injection   $\fo\hookrightarrow \End(A/S)$. 
The dual abelian variety of  $A$ is denoted by $A^\vee$.
 
A $\fc$-polarisation on a Hilbert-Blumenthal abelian variety  $A/S$ is an $\fo$-linear
isomorphism $$\lambda:A\otimes_{\fo}\fc \overset{\sim}{\longrightarrow} A^\vee$$
such that the  induced isomorphism $\Hom_{\fo}(A,A\otimes_{\fo}\fc)
\simeq \Hom_{\fo}(A,A^\vee)$ sends  $\fc$ (resp.\ $\fc_+$) onto the 
$\fo$-module of symmetric elements $\cP(A)$ (resp.\ onto the cone of polarisations $\cP(A)_+$).
 
A $\mu_{\fn}$-level structure on a Hilbert-Blumenthal abelian variety $A/S$
is an $\fo$-linear closed immersion $\mu:\mu_{\fn}\otimes \mathfrak{d}^{-1}  \hookrightarrow A$ of  group schemes over $S$, 
where  $\mu_{\fn}$ denotes the reduced  subscheme of $\mathbb{G}_m \otimes  \mathfrak{d}^{-1}$  defined
as the intersection of the kernels of multiplication by elements of $\fn$. Assume that
\begin{equation}\label{NT}
 \fn \text{ divides neither }2, \text{ nor } 3, \text{ nor the discriminant } \Nm(\fd). 
\end{equation}
Then the functor assigning to a $\Z[\frac{1}{\Nm(\fn)}]$-scheme $S$ the set of isomorphism classes
of tuples $(A,\lambda,\mu)$ as above is representable by a geometrically connected 
quasi-projective scheme $X^1_1(\fc,\fn)\to \Z[\frac{1}{\Nm(\fn)}]$, which is a relative complete intersection, flat and 
has normal fibres (see \cite[Theorem 2.2 and Corollary 2.3]{deligne-pappas}). We denote by $\pi:\cA(\fc)\rightarrow X^1_1(\fc,\fn)$ the universal abelian variety. 

Recall that the Rapoport locus $X^1_1(\fc,\fn)^\Ra$ of $X^1_1(\fc,\fn)$,  defined as the locus where  the sheaf
$\pi_* \Omega^1_{\cA(\fc)/X^1_1(\fc,\fn)}$ is locally free of rank one over  $\fo\otimes \cO_{X^1_1(\fc,\fn)}$, 
is an open sub-scheme, which is smooth. Its complement is contained in  
fibres in characteristics dividing $\Nm(\fd)$ and has codimension at least $2$ in those fibres.
In particular, $X^1_1(\fc,\fn)$ is smooth over $\Z[\frac{1}{\Nm(\fn\fd)}]$.

Consider the following invertible sheaf on $X^1_1(\fc,\fn)$:
$$\underline{\omega^1}:=\det( \pi_* \Omega^1_{\cA(\fc)/X^1_1(\fc,\fn)})=\wedge^d_{ \cO_{X^1_1(\fc,\fn)}} \pi_* \Omega^1_{\cA(\fc)/X^1_1(\fc,\fn)}.$$ 

Make the following, stronger than~\eqref{NT},  assumption: 
\begin{equation}\label{ass:21}
\begin{split}
& \text{$\fn$ is divisible by  a prime number $q$ which  splits completely in $F(\sqrt{\epsilon} \,|\,
\epsilon\in\fo_+^\times)$} \\
& \text{and $q\equiv -1\pmod{4\ell}$ for all prime numbers $\ell$ such that
$[F(\mu_\ell):F]=2$. }
\end{split}
\end{equation}
Under the assumption  ~\eqref{ass:21} the finite group 
\begin{equation}\label{units}
E:=\fo_+^\times/\{\epsilon\in \fo^\times| \epsilon-1\in \fn\}^2
\end{equation}
acts properly and discontinuously on
$X^1_1(\fc,\fn)$ by  $$[\epsilon]:(A,\lambda,\mu)/S\mapsto (A,\epsilon\lambda,\mu)/S,$$ 
yielding an étale morphism $\phi: X^1_1(\fc,\fn)\to X_1(\fc,\fn) $ with group $E$
(\cite[Lemma~2.1]{dim-ihara}). 
It follows that  $X_1(\fc,\fn)\to \Z[\frac{1}{\Nm(\fn)}]$ is  geometrically connected,  flat with normal fibres.
The similarly defined open subscheme $X_1(\fc,\fn)^\Ra$ of $X_1(\fc,\fn)$ is smooth and its complement is contained in  
fibres in characteristics dividing $\Nm(\fd)$ with codimension at least $2$ in those fibres.
Moreover  $\underline{\omega^1}$ descends to an invertible sheaf on  $X_1(\fc,\fn)$, denoted $\underline{\omega}$ (see \cite[\S4]{dimtildg}), such that its pull-back is~$\underline{\omega^1}$.

Let $\displaystyle X^1_1(\fn)=\coprod_{\fc} X^1_1(\fc,\fn) \text{ and } X_1(\fn)= \coprod_{\fc} X_1(\fc,\fn)$, 
where $\fc$ runs over the fixed set of representatives of $\cC\ell_F^+$.
The latter is called the Hilbert modular variety of level $\Gamma_1(\fn)$, since its complex points can be identified with those of a Shimura variety for $\GL_2(F)$. Note that the points of $X_1(\fn)$ over any $\Z[\frac{1}{\Nm(\fn)}]$-scheme $S$ are naturally given by isomorphism classes of  Hilbert-Blumenthal abelian varieties $A/S$ endowed with a $\mu_\fn$-level structure and satisfying the Deligne-Pappas condition that the induced morphism
$A \otimes_{\fo}\cP(A) \to A^\vee$ is an isomorphism.     
Thus $X_1(\fn)$ is independent of the choice of representatives in $\cC\ell_F^+$.

\subsection{Hilbert modular forms of parallel weight}\label{hmf}

Fix an integer~$k \ge 1$ and an ideal $\fn\subset \fo$. For any   $\Z[\frac{1}{\Nm(\fn)}]$-algebra~$R$ one can define as in Katz \cite{katz} the space of $\fc$-polarised Hilbert modular forms over $R$ 
of  weight $k$ and  level $\fn$.  As observed in \cite[\S1.4]{kisin-lai}, under the assumption 
(\ref{NT}) this space is given by: 
$$   \rH^0(X^1_1(\fc,\fn) \times \Spec(R), (\underline{\omega^1})^{\otimes k})
= \rH^0(X^1_1(\fc,\fn)^\Ra \times \Spec(R), (\underline{\omega^1})^{\otimes k}),$$  
where the equality follows from an algebraic version of Hartogs' theorem, since $X^1_1(\fc,\fn)$ is Noetherian with normal fibres,
$(\underline{\omega^1})^{\otimes k}$ is locally free and the complement has codimension at least~$2$ in each fibre.

In this paper, we will however reserve the name of Hilbert modular forms to those forms which are invariant under the action of the finite group $E$ defined in~(\ref{units}).
In fact, in characteristic~$0$, those are the only forms for which there is a satisfactory  Hecke theory allowing to attach Galois representations to eigenforms. 

Note that it suffices to prove Theorem~\ref{thm:main} under the assumption~\eqref{ass:21}, which we can and will always assume in the sequel. For a   $\Z[\frac{1}{\Nm(\fn)}]$-algebra~$R$, 
the space of Hilbert modular forms over $R$ of  weight $k$ and  level $\fn$  is then defined as the $R$-module:
\begin{equation}
M_k(\fn; R):= \rH^0(X_1(\fn) \times \Spec(R),\underline{\omega}^{\otimes k}).
\end{equation}

We have the following decomposition  $\displaystyle M_k(\fn; R)=\bigoplus_{\fc\in \cC\ell_F^+} M_k(\fc,\fn; R)$
with $\fc$ running over the fixed set of representatives of $\cC\ell_F^+$, where 
$$M_k(\fc,\fn; R)=\rH^0(X_1(\fc,\fn) \times \Spec(R),\underline{\omega}^{\otimes k}).$$

\subsection{Minimal compactification}

Following \cite[\S4]{chai} (see also \cite{dimdg}) one defines the minimal
compactification of $X^1_1(\fc,\fn)$ as the  $\Z[\frac{1}{\Nm(\fn)}]$-scheme given by
$$\overline{X^1_1(\fc,\fn)}= \mathrm{Proj}\left(\underset{k\geq 0}{\bigoplus}  
\rH^0(X^1_1(\fc,\fn),(\underline{\omega^1})^{\otimes k})\right).$$

The scheme $\overline{X^1_1(\fc,\fn)}$  is projective, normal, hence flat (see \cite[IV.6.8.1]{EGA}) over $\Z[\frac{1}{\Nm(\fn)}]$ and contains  $X^1_1(\fc,\fn)$  as a dense open sub-scheme the complement of which is finite over $\Z[\frac{1}{\Nm(\fn)}]$.

The action of the finite group  $E$ on $X^1_1(\fc,\fn)$ extends to an action on $\overline{X^1_1(\fc,\fn)}$ and the minimal compactification $\overline{X_1(\fc,\fn)}$ of $X_1(\fc,\fn)$ is defined as the quotient.  
The $\Z[\frac{1}{\Nm(\fn)}]$-scheme $\overline{X_1(\fc,\fn)}$ is projective, normal, flat over $\Z[\frac{1}{\Nm(\fn)}]$
and contains  $X_1(\fc,\fn)$  as a dense open sub-scheme.  Alternatively,
one can also define $\overline{X_1(\fc,\fn)}$ as $\mathrm{Proj}\left(\underset{k\geq 0}{\oplus}  \rH^0(X_1(\fc,\fn),\underline{\omega}^{\otimes k})\right)$.

\begin{lemma}
The  invertible sheaf   $\underline{\omega} $ on $X_1(\fc,\fn)$ extends to an ample  invertible sheaf  $\overline{\underline{\omega}}$  on  $\overline{X_1(\fc,\fn)}$.
\end{lemma}

\begin{proof}
As explained in \cite[\S1.8]{kisin-lai}, the arguments of \cite[\S4.4]{chai} go over unchanged to our setting and prove that the invertible sheaf $\underline{\omega^1}$ on $X^1_1(\fc,\fn)$ extends to an ample  invertible sheaf  $\overline{\underline{\omega^1}}$  on   $\overline{X^1_1(\fc,\fn)}$. 
By \cite[Theorem 8.6(vi)]{dimtildg},  $\underline{\omega}$ extends to an   invertible sheaf 
$\overline{\underline{\omega}}$  on  $\overline{X_1(\fc,\fn)}$. To show that $\overline{\underline{\omega}}$ 
is ample, consider the commutative diagram: 
$$\xymatrix{X^1_1(\fc,\fn)  \ar^{\phi}[d]   \ar[r] &\overline{X^1_1(\fc,\fn)} \ar^{\bar\phi}[d]\\
X_1(\fc,\fn)  \ar[r] &\overline{X_1(\fc,\fn)},}$$
where $\phi: {X^1_1(\fc,\fn)}\to {X_1(\fc,\fn)}$ is finite étale with group $E$, whereas 
  $\bar \phi: \overline{X^1_1(\fc,\fn)}\to \overline{X_1(\fc,\fn)}$ is a finite surjective morphism
of normal schemes. By \cite[II.6.6.2]{EGA} the norm  $\cN_{\bar\phi}(\overline{\underline{\omega^1}})$ is an 
ample invertible sheaf on $\overline{X_1(\fc,\fn)}$ (note that whereas condition (I) of {\em loc.\ cit.} may not hold since $\bar\phi$ is not necessarily flat, condition (IIbis) is always fulfilled since $\overline{X_1(\fc,\fn)}$
is  normal). Moreover one has
$$\cN_{\bar\phi}(\overline{\underline{\omega^1}})_{|X_1(\fc,\fn)}= \cN_{\phi}(\underline{\omega^1})=
\cN_{\phi}(\phi^*{\underline{\omega}})= \underline{\omega}^{\otimes |E|},$$
where the final equality follows from~\cite[II.6.5.2.4]{EGA}.
Since  $\overline{X_1(\fc,\fn)}$ is normal and the complement of $X_1(\fc,\fn)$ has codimension $d\geq 2$, 
one deduces that 
$\cN_{\bar\phi}(\overline{\underline{\omega^1}})=\overline{\underline{\omega}}^{\otimes |E|}$. 
Therefore $\overline{\underline{\omega}}^{\otimes |E|}$ is ample, which implies that 
$\overline{\underline{\omega}}$ is ample too. 
\end{proof}

The complement of $X_1(\fc,\fn)$ in $\overline{X_1(\fc,\fn)}$ consists of closed points, called cusps. 
Amongst these one has the `infinity' cusps  $\infty(\fc;\fb):\Spec(\Z[\frac{1}{\Nm(\fn)}])\to \overline{X_1(\fc,\fn)}$, 
with $\fb$ running over a set of representatives of $\cC\ell_F$, the standard infinity cusp being $\infty(\fc)=\infty(\fc;\fo)$.
By \cite[\S4]{dimdg}  the inverse  image of $\infty(\fc,\fb)$ under $\bar\phi$ in $\overline{X_1^1(\fc,\fn)}$ consists of a unique point
and the Tate object at that cusp is given by  $(\mathbb{G}_m\otimes_\fo \fb^{-1}\fc^{-1}\fd^{-1})/q^{\fb}$. 

Let $R$ be a $\Z[\frac{1}{\Nm(\fn)}]$-algebra. 
The Koecher Principle (see \cite[4.9]{rapoport} and \cite[Theorem 8.3]{dimdg}) implies that the natural restriction map 
\begin{equation}\label{koecher}
\rH^0(\overline{X_1(\fc,\fn)}\times\Spec(R), \overline{\underline{\omega}}^{\otimes k})\overset{\sim}{\longrightarrow}
\rH^0(X_1(\fc,\fn) \times\Spec(R), \underline{\omega}^{\otimes k})=M_k(\fc;\fn;R).
\end{equation}
is an isomorphism. Note that the arguments from {\em loc.\ cit.} are purely at~$\infty$,
hence also apply to $R$ in which the discriminant is not necessarily invertible.

\subsection{Ampleness and lifting}
Fix a prime number $p$, relatively prime to $\fn$. 

The construction of a Galois representation for a normalised Hilbert modular eigenform of  weight $1$ over~$\Fbar_p$ is achieved by constructing a lifting in characteristic $0$ in some, possibly higher,   weight. 
The  existence of such lifts with special properties relative to the Hecke operators $T_\fp$ at 
primes $\fp$ dividing $p$,  namely being ordinary,  is an important ingredient in the proof of our main theorem, allowing us to apply  theorems of Hida and Wiles  describing the local behaviour at $\fp$
of the $p$-adic Galois representations attached to those lifts. 
Contrary to the situation for elliptic modular forms, we are not aware of a precise way to control the weight of the lift.
If the weight is at least~$3$, Stroh and Lan-Suh have proved independently that one can lift cusp forms under the condition that $p$ is at least~$d$  and does not divide $\Nm(\fn\fd)$ (see \cite{lan-suh}).

We now prove the lifting result we need with a method going back to Katz based on Serre's vanishing theorem in coherent cohomology
(compare with  \cite[11.9]{andreatta-goren}). 

\begin{lemma}\label{lem:lift}
There exists an integer $k_0$ depending on $\fn$, such that for all $k>k_0$
there is a natural Hecke equivariant isomorphism: 
$$M_k(\fn ; \Z_p) \otimes \F_p \simeq M_k(\fn ; \F_p).$$
\end{lemma}

\begin{proof} 
Fix a fractional ideal $\fc$ and let $X=X_1(\fc,\fn)\times\Spec(\Z_p)$. 
The projective variety  $\overline{X}= \overline{X_1(\fc,\fn)}\times\Spec(\Z_p)$ is normal, hence flat over $\Z_p$. 
Since $\overline{\underline{\omega}}^{\otimes k}$ is a locally free  $\cO_{\overline{X}}$-module, one gets a short exact sequence of sheaves on $\overline{X}$:
$$0\to \overline{\underline{\omega}}^{\otimes k} \overset{p\cdot}{\longrightarrow}
\overline{\underline{\omega}}^{\otimes k} \to 
\overline{\underline{\omega}}^{\otimes k}_{\F_p} \to 0,$$
yielding, using Koecher's principle \eqref{koecher}, a long exact sequence in cohomology: 
$$0\to M_k(\fc, \fn ; \Z_p)\overset{p\cdot}{\longrightarrow} 
M_k(\fc, \fn ; \Z_p) \to \rH^0(\overline{X}, \overline{\underline{\omega}}^{\otimes k}_{\F_p})
\to \rH^1(\overline{X}, \overline{\underline{\omega}}^{\otimes k}). $$

Since $\overline{\underline{\omega}}^{\otimes k}_{\F_p}$ has its support in 
$\overline{X}\times\Spec(\F_p)$, and using again \eqref{koecher} one gets
$$\rH^0(\overline{X}, \overline{\underline{\omega}}^{\otimes k}_{\F_p})=
\rH^0(\overline{X}\times \Spec(\F_p), \overline{\underline{\omega}}^{\otimes k}_{\F_p})=
M_k(\fc, \fn ; \F_p).$$

Finally by a theorem due to Serre \cite[III.2.2.1]{EGA}, since $\overline{X}$ is projective over $\Z_p$ and 
$\overline{\underline{\omega}}$ is an ample invertible sheaf, there exists an integer $k_0$, such that $\rH^1(\overline{X}, \overline{\underline{\omega}}^{\otimes k})$ vanishes for all $k>k_0$.  
Letting $\fc$ run over  the fixed set of representatives of $\cC\ell_F^+$ yields the desired result. 
\end{proof}

\subsection{Geometric $q$-expansion}\label{q-exp}
Let $R$ be a $\Z[\frac{1}{\Nm(\fn)}]$-algebra. 
By \cite[Theorem 8.6]{dimdg}  the local completed ring of $\overline{X_1(\fc,\fn)}\times\Spec(R)$ along~$\infty(\fc)$ is given by
\begin{equation}
M_\infty(\fc;R):=\left\{\sum_{\xi\in\fc_+\cup\{0\}} a_\xi q^\xi \;|\; a_\xi\in R  \text{ and }  a_{\epsilon\xi}= a_\xi,  \forall \epsilon \in  \fo_+^\times\right\}.
\end{equation}
Furthermore by {\em loc.\ cit.} the local completed ring of $\overline{X_1(\fc,\fn)}\times\Spec(R)$ along~$\infty(\fc;\fb)$ is given by $M_\infty(\fc\fb^2;R)$.

Since  $\overline{\underline{\omega}}$ can be trivialised over the formal neighbourhood 
$\mathcal{X}(\fc):=\Spf(M_\infty(\fc;R))$ of $\infty(\fc)$,  (\ref{koecher}) yields the {\em geometric $q$-expansion}  map
\begin{equation}
M_k(\fc,\fn; R) \to M_\infty(\fc;R), \enspace f_\fc \mapsto \sum_{\xi\in\fc_+\cup\{0\}} a_\xi(f) q^\xi,
\end{equation}
which due to its local definition factors through the map 
\begin{equation}\label{eq:q-exp}
\rH^0(U, \overline{\underline{\omega}}^{\otimes k}) \to M_\infty(\fc;R),
  \enspace f \mapsto f_\fc =\sum_{\xi\in\fc_+\cup\{0\}} a_\xi(f_\fc) q^\xi,
\end{equation}
for any  open subscheme $U\to \Spec(R)$ of $\overline{X_1(\fc,\fn)}\times\Spec(R)$ containing $\infty(\fc)$.

\begin{prop}[$q$-expansion principle]\label{prop:q-exp}
The map \eqref{eq:q-exp} is injective.
Moreover, for any $\Z[\frac{1}{\Nm(\fn)}]$-subalgebra $R'\subset R$ and
any  open subscheme $U\to \Spec(R')$ of $\overline{X_1(\fc,\fn)}$ containing $\infty(\fc)$, one has 
$$ \rH^0(U \times_{\Spec(R')}\Spec(R), \overline{\underline{\omega}}^{\otimes k}) \cap M_\infty(\fc;R')=
   \rH^0(U, \overline{\underline{\omega}}^{\otimes k}).$$ 
In particular $ M_k(\fc,\fn;R)\cap M_\infty(\fc;R')=M_k(\fc,\fn; R')$. 
\end{prop}

\begin{proof} Since the geometric fibres of $X_1(\fc,\fn)\to \Spec(\Z[\frac{1}{\Nm(\fn)}])$ are normal,  the argument of  \cite[Theorem 6.1]{rapoport} implies that they are  irreducible. Hence the geometric fibres of $U\to \Spec(\Z[\frac{1}{\Nm(\fn)}])$ are irreducible too and the above claims
follow by exactly the same arguments as in \cite[Theorem 6.7]{rapoport}. 
\end{proof}

\subsection{Adelic $q$-expansion}\label{adelic-q-exp}

Let $f\in M_k(\fn;R)$. By Proposition \ref{prop:q-exp}, $f$  is determined by its $q$-expansions $f_\fc\in  M_\infty(\fc;R)$ at $\infty(\fc)$, where $\fc$ runs over the  fixed set of representatives of $\cC\ell_F^+$.
We now define the adelic $q$-expansion of $f$, following \cite[\S2]{shimura} (see also \cite{hida-PAF}).
Let
$$M(R):=  R[\cC\ell_F^+] \bigoplus\left\{\sum_{\fb \in \cI} a(\fb) q^\fb \;|\; a(\fb) \in R\right\},$$
where $\cI$ is the group of fractional ideals of~$F$.
For any fractional ideal~$\fc$, consider the adelic expansion map
$$\psi_\fc: M_\infty(\fc;R) \to M(R)$$ 
sending $\sum_{\xi\in\fc_+\cup\{0\}} a_\xi q^\xi$ to the element of $\sum_{\fb \in \cI\cup\{(0)\}} a(\fb) q^\fb \in M(R)$ such that 
$$ a(\fb) = \begin{cases}
a_\xi     & \textnormal{ if } \fb \fc = (\xi) \textnormal{ for some }\xi \in \fc_+,\\
a_0[\fc]  & \textnormal{ if }\fb = (0),\\
0            & \textnormal{ otherwise,}
\end{cases}$$
where the first alternative means that $\fb\fc$ is the trivial element in $\cC\ell_F^+$ so that $\fb \fc$ is
principal with some totally positive generator~$\xi$, uniquely determined up to totally positive units.
The invariance property of $M_\infty(\fc;R)$ implies that this map does not depend on the choice of the totally positive generator $\xi$ of the ideal $\fb \fc$.
This hints to us why taking $E$-invariants is necessary in order to have a good theory of Hecke operators on 
$q$-expansions.

Let $\xi \in F_+^\times$ so that $\fc$ and $\fc'=\xi\fc $ are in the same narrow ideal class. Then
$$ \psi_\xi: M_\infty(\fc;R) \to M_\infty(\fc';R), \;\;\; \sum_{\xi'\in\fc_+\cup\{0\}} a_{\xi'} q^{\xi'} \mapsto \sum_{\xi'\in\fc_+\cup\{0\}} a_{\xi'} q^{\xi\xi'}$$
is an isomorphism of $R$-algebras and we have
$\psi_{\fc'} \circ \psi_\xi = \psi_\fc$.

Adding up the maps $\psi_\fc$, gives a well-defined map, the adelic $q$-expansion:
$$ \psi: \bigoplus_{\fc \in \cC \ell_F^+} M_\infty(\fc;R) \to M(R),$$
where $\fc$ runs over  the fixed set of representatives of the narrow ideal class group.
The map $\psi$ is independent of the choice of these representatives in the above sense.

Note that while  any $f\in M_k(\fn;R)$ is determined by its adelic $q$-expansion,
the latter does not have a direct geometric interpretation if $F\neq \Q$. 

Let $\fc$ be a fractional ideal and $\fq $ be an integral ideal. It is immediate to check that the
inclusion $\iota_\fq : M_\infty(\fc \fq ; R) \hookrightarrow M_\infty(\fc; R)$ corresponds to the
$\fq $-power map on adelic $q$-expansions, {\it i.e.,}\ $a(\fb, f) = a(\fb\fq , \iota_\fq (f))$ for all integral ideals~$\fb$.
Furthermore,  the normalised trace map
$ t_\fq : M_\infty(\fc; R) \to M_\infty(\fc\fq ; R)$
given by keeping only those terms which are indexed by elements of $\fc\fq$, corresponds to the $1/\fq $-power map on adelic $q$-expansions, {\it i.e.,}\ $a(\fb\fq , f) = a(\fb, t_\fq (f))$ for all integral ideals~$\fb$.

\section{Hecke operators}\label{operators}

In this section we provide the operators on geometric Hilbert modular forms that are needed to prove the main theorem.
In particular, we recall the constructions of the diamond operators and of Hecke operators of index coprime to the level and the characteristic.
The core of this section is the definition of the Hecke operator $T_\fp$ for $\fp \mid p$ and the computation of its
effect on adelic $q$-expansions.
It is also the main input in our definition of Frobenius operators.

\subsection{Diamond operators}\label{diamond}

Let $\fq$ be any ideal of $\fo$ relatively prime to $\fn$.
Recall that by definition one has an exact sequence 
$$ 0\to (A\otimes_\fo\fq)[\fq] \to A\otimes_\fo\fq \overset{i}{\to}  A\to 0,$$
yielding, since  by  \cite[\S1.9]{kisin-lai} the Cartier dual of $A^\vee[\fq]$ is $ (A\otimes_\fo\fq)[\fq]$, 
$$ 0\to A^\vee[\fq] \to A^\vee   \overset{i^\vee}{\to}   (A\otimes_\fo\fq)^\vee \to 0.$$
It results from these exact sequences that there is a canonical isomorphism
\begin{equation}\label{cartier}
 A^\vee\otimes_\fo\fq^{-1}\overset{\sim}{\longrightarrow}  (A\otimes_\fo\fq)^\vee
\end{equation}

Consider the automorphism $\langle\fq\rangle$ of $X^1_1(\fn)$ sending $(A,\lambda,\mu)$
to $(A\otimes_\fo \fq,\lambda',\mu')$, where $\mu'$ is the composite of $\mu$ with the inverse of the canonical 
isomorphism $(A\otimes_\fo\fq)[\fn]\simeq A[\fn]$ induced by $ A\otimes_\fo\fq \to A$, and $\lambda'$ is the following
$\fc\fq^{-2}$-polarisation on $A\otimes_\fo\fq$: 
$$\lambda': (A\otimes_\fo\fq)\otimes_\fo\fc\fq^{-2}\underset{\sim}{\overset{\lambda\otimes \fq^{-1}}{\longrightarrow}}
A^\vee\otimes_\fo\fq^{-1}\underset{\sim}{\overset{\eqref{cartier}}{\longrightarrow}} (A\otimes_\fo\fq)^\vee.$$

Since $\langle\fq\rangle$ commutes with the $E$-action, it induces an automorphism of  $X_1(\fn)$ extending uniquely 
to an automorphism of $\overline{X_1(\fn)}$. 
For every $\Z[\frac{1}{\Nm(\fn)}]$-algebra~$R$ the resulting  $R$-linear endomorphism $\langle\fq\rangle$ 
of $M_k(\fn; R)$  sending $M_k(\fc\fq^{-1},\fn; R)$ to $M_k(\fc\fq,\fn; R)$ is called 
the diamond operator (note that this operator is the inverse of the one defined by Hida in \cite[4.1.9]{hida-PAF}).

\begin{lemma}\label{formal-diamond}
The morphism $\langle\fq\rangle$  sends the cusp $\infty(\fc\fq,\fq^{-1})$  to $\infty(\fc\fq^{-1})$
and identifies the formal completed rings  at those cusps. 
In particular, the  geometric $q$-expansion of any  $f\in M_k(\fc\fq,\fn; R)$  at  $\infty(\fc\fq,\fq^{-1})$ 
equals $(\langle\fq^{-1}\rangle f)_{\fc\fq^{-1}}$. 
\end{lemma}

\begin{proof}
The Tate object at $\bar\phi^{-1}(\infty(\fc\fq^{-1}))$ is $A=(\mathbb{G}_m\otimes_\fo \fc^{-1}\fd^{-1}\fq)/q^{\fo}$
and its image under $\langle\fq^{-1}\rangle$  equals  $A\otimes_\fo \fq^{-1}\simeq A/A[\fq]\simeq
(\mathbb{G}_m\otimes_\fo \fc^{-1}\fd^{-1})/q^{\fq^{-1}}$, hence $\langle\fq^{-1}\rangle$ sends  $\infty(\fc\fq^{-1})$ to 
$\infty(\fc\fq,\fq^{-1})$. Note that the formal completed rings at  both cusps are given by
$M_\infty(\fc\fq^{-1}, R)$,  hence the claim. 
\end{proof}

It can be checked that the diamond operators induce an action of the ray class group of modulus~$\fn$ on $M_k(\fn;R)$.
Given any $R$-valued finite order Hecke character $\varepsilon$ of $F$ of conductor dividing $\fn$, we  let  $M_k(\fn,\varepsilon; R)$ denote the sub-$R$-module
of $M_k(\fn; R)$ of forms on which $\langle\fq\rangle$ acts as $\varepsilon(\fq)$ for all ideals $\fq$ relatively prime to~$\fn$.

\subsection{Hecke operators away from the characteristic}\label{sec:Tq}

Let $\fq$ be a prime of $F$ which is relatively prime to $\fn$, and let  $R$ be a $\Z[\frac{1}{\Nm(\fn\fq)}]$-algebra. 
The assumption that $\fq$ be invertible in~$R$ is removed in the next section when we construct the operators $T_\fp$ for $\fp\mid p$
in characteristic~$p$. 

Let $X=X_1(\fn)\times\Spec(R)$ and denote by $\overline{X}$ its minimal compactification, which is a
normal projective scheme over $R$. One can define as in \S\ref{hmv}  a Hilbert modular variety 
$Y$ of level $\Gamma_1(\fn)\cap\Gamma_0(\fq)$  over  $R$ endowed with 
two finite surjective morphisms  $\pi_1,\pi_2: Y \rightrightarrows X$, corresponding to the forgetful and the `quotient' maps
on the level of moduli spaces, respectively. 
A careful inspection of the moduli theoretic definition of $\pi_1$ and $\pi_2$ (see  \cite[\S1.10]{kisin-lai})  shows that  they  factor through one another. Therefore their extensions  to the minimal compactifications  fit in a  commutative diagram: 
\begin{equation}\label{atkin-lehner}
\xymatrix{\overline{Y} \ar_{\pi_1}[rd]  \ar[rr]^{\iota}_{\sim}&  &\overline{Y}\ar^{\pi_2}[ld]\\   & \overline{X} & },
\end{equation}
where $\iota$ is an automorphism, but not necessarily an involution,   of  $\overline{Y}$.

The Hecke operator $T_{\fq}^{(k)\vee}$ considered in \cite[\S2.4]{dim-hmv} is the $R$-linear endomorphism 
of $M_k(\fn; R)$ defined as the composite of the following three maps: 

\begin{enumerate}
\item  the inclusion coming from the adjunction morphism:
$$\rH^0(\overline{X}, \overline{\underline{\omega}}^{\otimes k})\to 
\rH^0(\overline{X}, \pi_{2*}\pi_2^*\overline{\underline{\omega}}^{\otimes k})= 
\rH^0(\overline{Y}, \pi_2^*\overline{\underline{\omega}}^{\otimes k});$$
\item the morphism 
\begin{equation}\label{connecting-morphism}
\rH^0(\overline{Y}, \pi_2^*\overline{\underline{\omega}}^{\otimes k})\to 
\rH^0(\overline{Y}, \pi_1^*\overline{\underline{\omega}}^{\otimes k})
\end{equation}
 induced by a morphisms of sheaves
$\pi_2^*\overline{\underline{\omega}}^{\otimes k}\to \pi_1^*\overline{\underline{\omega}}^{\otimes k}$ (see \cite[(1.11.1)]{kisin-lai});
\item the morphism
$\rH^0(\overline{Y}, \pi_1^*\overline{\underline{\omega}}^{\otimes k})= \rH^0(\overline{X},\pi_{1*} \pi_1^*\overline{\underline{\omega}}^{\otimes k})\to \rH^0(\overline{X}, \overline{\underline{\omega}}^{\otimes k})$
induced by the trace map  $\pi_{1*} \pi_1^*\overline{\underline{\omega}}^{\otimes k}\to \overline{\underline{\omega}}^{\otimes k}$ relative to the finite morphism $\pi_1$, followed by the multiplication by $\Nm(\fq)^{-1}\in R$.
\end{enumerate}
Henceforth  we set  $T_{\fq}^{(k)}=T_{\fq}^{(k)\vee}\circ \langle\fq\rangle$ (this is the normalisation used by Gross  \cite{gross}).

When $R$ has characteristic $0$, comparison with the complex analytic theory of automorphic forms for $\GL_2(F)$ implies that for all $f\in M_k(\fn;R)$ and for all non-zero ideals $\fr\subset \fo$ one has (see~\cite{shimura}): 
\begin{equation}\label{Tq}
 a(\fr, T_\fq^{(k)} f) = a(\fq \fr,f) + \Nm(\fq)^{k-1} a(\fr/\fq,\langle\fq\rangle f)
\end{equation}
where as usual $a(\fr/\fq,f)=0$ if $\fq \nmid \fr$. When $\fr=(0)$ the following equality holds in $R[\cC\ell_F^+]$:
\begin{equation}\label{eq:Tq0}
 a((0), T_\fq^{(k)} f) = a((0), f)[\fq] + \Nm(\fq)^{k-1}a((0), \langle\fq\rangle f)[\fq^{-1}].
\end{equation}
In fact, since \eqref{eq:Tq0}  is clearly satisfied by cusp forms, it suffices to check it on the  Eisenstein subspace which has a  basis consisting of Eisenstein series attached to pairs of 
Hecke character  $\varphi$ and $\varepsilon\varphi^{-1}$, where $\varepsilon$  is the central character. Moreover  the constant term at the~$\infty$ cusps vanishes unless  $\varphi$ is a character of $\cC\ell_F^+$ (see \cite[\S 2.2]{DDP}), in which case one has
\begin{equation*}
 a((0), T_\fq^{(k)} f) =(\varphi(\fq) + \varepsilon(\fq) \varphi^{-1}(\fq))a((0), f)=
 a((0), f)[\fq] + \varepsilon(\fq)\Nm(\fq)^{k-1}a((0), f)[\fq^{-1}].
\end{equation*}

Together with the $q$-expansion principle,   formulas  \eqref{Tq} and \eqref{eq:Tq0}
imply that the $T_\fq^{(k)}$s commute with each other. Moreover if $R$ is a 
characteristic $0$ field then   $T_\fq^{(k)}$ is semi-simple.

Formulas \eqref{Tq} and \eqref{eq:Tq0} remain valid when $R$ has positive characteristic relatively prime to $\fq$,
which can also be seen by (a simpler version of) the argument given in the next section.

\subsection{Hecke operators in characteristic $p$}

Fix a prime number $p$ which is relatively prime to $\fn$  and fix a  prime $\fp$ dividing~$p$. 
In  \S\ref{sec:Tq} we defined a $\Q_p$-linear endomorphism $T_{\fp}^{(k)}$
of $M_k(\fn; \Q_p)$.  It follows  from (\ref{Tq}) and the $q$-expansion principle that $T_{\fp}^{(k)}$ induces a  $\Z_p$-linear endomorphism of $M_k(\fn; \Z_p)$, hence a  $\F_p$-linear endomorphism of $M_k(\fn; \Z_p)\otimes \F_p$. The aim of this section is to extend  $T_\fp^{(k)}$ to an endomorphism of the whole space $M_k(\fn;\F_p)$ for all $k\geq 1$. Note that such an endomorphism would be uniquely determined  by its action on  $q$-expansions.  

Emerton, Reduzzi and Xiao propose in  \cite[\S 3]{ERX}  
a different approach in the derived category using the dualising trace map. 

An original feature in our approach is the detailed study of the formal skeleton of the $T_\fp$-correspondence, allowing us to 
compute the effect of~$T_\fp^{(k)}$ on $q$-expansions of modular forms. This is based on an isomorphism between 
the (minimal compactification of the) Hilbert modular variety $\overline{Y(\fc)}$ with $\Gamma_0(\fp)$-level structure
and a suitable relative normalisation (see Lemma \ref{kwlan}) allowing us to determine its formal completion at the cusps
$\infty_\fc$ and $0_\fc$ in Corollary~\ref{formal-completion}.

\subsubsection{Integral models by normalisation}

For any fractional ideal $\fc$, let $X(\fc)=X_1(\fc,\fn)\times\Spec(\Z_p)$. Its minimal compactification
$\overline{X(\fc)}$  is a  flat projective scheme over $\Z_p$, the fibres of which are normal and geometrically irreducible.

The scheme $Y(\fc)_{\Q_p}$ considered in the previous section for $R=\Q_p$ admits a
model $Y(\fc)$ over $\Z_p$ which is normal and is a relative complete intersection. Namely one can define $Y(\fc)$ as  the  quotient by $E$ of the moduli space with $\Gamma_0(\fp)$-level structure $Y^1(\fc)$ defined by Stamm and Pappas \cite{pappas1995}. 
Its minimal compactification $\overline{Y(\fc)}$ is proper over $\Z_p$ and normal.
  
By construction (see \cite[\S1.9]{kisin-lai}) there exist degeneracy morphisms
$$\pi_1 : \overline{Y(\fc)}\to  \overline{X(\fc)} \text{ and } \pi_2 : \overline{Y(\fc)}\to  \overline{X(\fc\fp)}, $$ 
which are proper (see ~\cite[Tag 01W6]{stacks-project}), surjective and  generically finite, but may not be  finite over the non-ordinary locus in the special fibre (see \cite{goren-kassaei}).
In particular, since all schemes are separated,
the  morphism $(\pi_1, \pi_2) : \overline{Y(\fc)}\to  \overline{X(\fc)} \times \overline{X(\fc\fp)}$ is proper, too.

Define $\overline{Z(\fc)'}$ (resp.\  $\overline{Z(\fc)''}$) as the relative normalisation of  
$$\overline{Y(\fc)}_{\Q_p} \overset{\pi_1}{\longrightarrow}  \overline{X(\fc)}_{\Q_p}\to \overline{X(\fc)} \;
 \big(\text{resp.\ } \overline{Y(\fc)}_{\Q_p} \overset{\pi_2}{\longrightarrow}  \overline{X(\fc\fp)}_{\Q_p}\to \overline{X(\fc\fp)}\;\big).
$$ 
Note that while the morphism $\pi_1: \overline{Z(\fc)'}\to \overline{X(\fc)}$ is finite, one cannot in general define a morphism $ \overline{Z(\fc)'}\to \overline{X(\fc\fp)}$ extending $\pi_2:  \overline{Y(\fc)}_{\Q_p}\to   \overline{X(\fc\fp)}_{\Q_p}$. To palliate this problem we consider the relative normalisation $\overline{Z(\fc)}$  of 
\begin{equation}
\overline{Y(\fc)}_{\Q_p} \overset{(\pi_1,\pi_2)}{\longrightarrow}  \overline{X(\fc)}_{\Q_p}\times
 \overline{X(\fc\fp)}_{\Q_p} \to \overline{X(\fc)}\times \overline{X(\fc\fp)}.
 \end{equation} 
By functoriality of relative  normalisation  (e.g.~\cite[Tag 035J]{stacks-project}) there exist canonical 
commutative diagrams: 
$$  
\xymatrix{ \overline{Z(\fc)}  \ar_{\pi'''_1}[rd]  \ar^{\nu_1}[r] &\overline{Z(\fc)'}\ar^{\pi'_1}[d]\\
    & \overline{X(\fc)} }  \text{, and } 
 \xymatrix{   \overline{Z(\fc)}  \ar_{\pi'''_2}[rd]  \ar^{\nu_2}[r] &\overline{Z(\fc)''}\ar^{\pi''_2}[d]\\
  &  \overline{X(\fc\fp)} },
$$
where the horizontal morphisms are proper (e.g.~\cite[Tag 01W6]{stacks-project}) and birational,  as they induce isomorphisms on dense $\Q_p$-fibres. 
Since the schemes $\overline{Z(\fc)}$, $\overline{Z(\fc)'}$ and $\overline{Z(\fc)''}$ are normal (e.g.~\cite[Tag 035L]{stacks-project}), it follows from \cite[Tag 0AB1]{stacks-project}) that $\nu_i$ becomes an 
isomorphism when restricted to any open over which $\pi'''_i$ is finite,  in particular over the ordinary locus (see \cite[Proposition 3.7]{ERX}). 

Finally we prove a useful lemma by reproducing an argument communicated to us by Kai-Wen Lan. 

\begin{lemma}\label{kwlan}
The morphism $(\pi_1, \pi_2)  : \overline{Y(\fc)}\to  \overline{X(\fc)} \times  \overline{X(\fc\fp)}$ is finite and there
is a canonical isomorphism between $\overline{Z(\fc)}$ and $\overline{Y(\fc)}$. 
\end{lemma}

\begin{proof}
Suppose first that $(\pi_1, \pi_2)$ is integral. Then by the universal property of normalisation (e.g.~\cite[Tag 035I]{stacks-project}) 
there exists a unique morphism $\nu$ inducing identity on dense 
$\Q_p$-fibres and fitting in the following commutative diagram
$$\xymatrix{ \overline{Z(\fc)}  \ar_{(\pi'''_1,\pi'''_2)}[rd]  \ar^{\nu}[r] &\overline{Y(\fc)}\ar^{(\pi_1,\pi_2)}[d]\\
    & \overline{X(\fc)} \times  \overline{X(\fc\fp)}. } $$
Since  $(\pi'''_1,\pi'''_2)$ is finite by  construction (e.g.~\cite[Tag 01WJ]{stacks-project}), one deduces that
$\nu$ is finite, too (e.g.~\cite[Tag 035D]{stacks-project}), which together with the fact that 
$\overline{Y(\fc)}$ is normal, implies that  $\nu$ is an isomorphism (see \cite[Tag 0AB1]{stacks-project}).  
   
To prove that $(\pi_1, \pi_2)$ is finite,  by properness it suffices to show that it is quasi-finite,
for which one can restrict to the non-cuspidal locus since the  cuspidal locus in $\overline{Y(\fc)}$ consists only of closed points. We will  establish the finiteness of the morphism 
$(\pi_1, \pi_2)  : Y^1(\fc)\to  X^1(\fc) \times  X^1(\fc\fp)$ on the level of fine moduli spaces, 
since taking the quotient by the finite group $E$ preserves this property. 
  Let $Z^1(\fc)$ be  the relative normalisation of 
\begin{equation}
 Y^1(\fc)_{\Q_p} \overset{(\pi_1,\pi_2)}{\longrightarrow}  X^1(\fc)_{\Q_p}\times
 X^1(\fc\fp)_{\Q_p} \to X^1(\fc)\times X^1(\fc\fp) 
\end{equation} 
and consider the pullbacks  $(\pi'''_1)^* \cA(\fc) $ and $(\pi'''_2)^* \cA(\fc\fp) $ to $Z^1(\fc)$  of the  universal abelian schemes over  $X^1(\fc)$ and $X^1(\fc\fp)$ defined in \S\ref{hmv}. Since $Z^1(\fc)$ is  noetherian and normal, Raynaud's theorem 
(see~\cite[Cor.~IX.1.4]{raynaud} or \cite[Prop.~I.2.7]{faltings-chai}) implies that the universal isogeny between 
the  $\Q_p$-fibres of  these two pullbacks  extends, yielding by the universal property of the fine moduli space $Y^1(\fc)$ a commutative diagram
 $$\xymatrix{ Z^1(\fc)  \ar_{(\pi'''_1,\pi'''_2)}[rd]  \ar^{\nu}[r] & Y^1(\fc) \ar^{(\pi_1,\pi_2)}[d]\\
    & X^1(\fc)\times X^1(\fc\fp). } $$
By the same arguments as above, the finiteness of  $ (\pi'''_1,\pi'''_2)$ and the normality of $Y^1(\fc)$ imply that    $\nu$ is an isomorphism and that $(\pi_1, \pi_2)$ is finite, finishing the proof of the lemma.      
\end{proof}

\subsubsection{The formal skeleton of the $T_\fp$ correspondence}
As the $T_\fp$-correspondence does not preserve individual components, we consider
$X=\coprod_{\fc\in \cC\ell_F^+} X(\fc)$ and $Y=\coprod_{\fc\in \cC\ell_F^+} Y(\fc)$.
Since we will be mostly interested in the effect of $T_\fp$ on $q$-expansions, it is natural to study the  pull-back of $\pi_1$ and $\pi_2$ to  formal neighbourhoods of the  infinity cusps of $\overline{X}$.  Since those cusps belong to the ordinary locus, by the discussion in the previous  paragraph one can use the definition  $\overline{Y}$ based on  normalisations. 

\begin{prop}\label{formal-cusps}
The inverse image under $\pi_1$ of the  cusp $\infty(\fc):\Spec(\Z_p)\to \overline{X(\fc)}$ consists of two cusps
denoted $\infty_{\fc}, 0_{\fc}:\Spec(\Z_p)\to \overline{Y(\fc)}$, and  the formal completion  of $\overline{Y(\fc)}$ along
$\infty_{\fc}$ (resp.\  $0_{\fc}$)   is given by 
$$\mathcal{Y}_\infty(\fc)= \Spf(M_\infty(\fc;\Z_p))\text{ (resp.\ }
\mathcal{Y}_0(\fc) = \Spf(M_\infty(\fc\fp^{-1};\Z_p))).$$
\end{prop}

\begin{proof} 
By \cite[Theorem 8.6]{dimdg},  the formal completion  $\mathcal{X}(\fc)$ of $\overline{X(\fc)}$ at $\infty(\fc)$ is given by $\Spf(M_\infty(\fc;\Z_p))$. 
Also by {\em loc.\ cit.}  the local completed ring of $\overline{Y(\fc)}_{\Q_p}$ at $\infty_{\fc}$ (resp.\ $0_{\fc}$) is given by $M_\infty(\fc;\Q_p)$ (resp.\   $M_\infty(\fc\fp^{-1};\Q_p)$). 

The scheme $\overline{X(\fc)}$ is of finite type over $\Z_p$, hence excellent by \cite[IV.7.8.6]{EGA}. 
Since by \cite[IV.7.8.3(vii)]{EGA} normalisation  and completion commute for reduced
excellent local rings, it follows from the definition of $\overline{Z(\fc)}$
that its local completion at $\pi_1^{-1}(\infty(\fc))$ is given by 
the normalisation of $M_\infty(\fc;\Z_p)$ in $M_\infty(\fc;\Q_p)  \times M_\infty(\fc\fp^{-1};\Q_p)$. 
In particular, it follows that $\pi_1^{-1}(\infty(\fc))$ consists of two points. 

Since for any $\fc'\supset \fc$ the ring $M_\infty(\fc';\Z_p) $ is a finitely generated module and hence integral over $M_\infty(\fc;\Z_p)$, it  remains to show that any element of $M_\infty(\fc';\Q_p) $ which is integral over $M_\infty(\fc;\Z_p)$, belongs to $M_\infty(\fc';\Z_p)$. 
 
Since $\mathcal{X}(\fc')$ is normal and excellent its completed local ring at infinity $M_\infty(\fc';\Z_p)$ is a normal ring, 
hence  it is integrally closed in $M_\infty(\fc';\Q_p)$.
Since the embedding $M_\infty(\fc;\Z_p) \to M_\infty(\fc';\Q_p)$ factors through $M_\infty(\fc';\Z_p)$,
any element of $M_\infty(\fc';\Q_p)$ which is integral over $M_\infty(\fc;\Z_p)$ is in particular integral
over $M_\infty(\fc';\Z_p)$, hence belongs to the latter ring, as was to be shown.

Another way to see this is to use that $M_\infty(\fc;\Z_p)$ is isomorphic to the $\fo_+^\times$-invariants in  an intersection of power series rings $R_\sigma^\wedge(\fc;\Z_p)$ as in \cite[\S2]{dimdg}, where $\sigma$ runs over a rational cone decomposition of $F_+^\times$. The claim then follows from the well-known analogous statement for polynomial and power series rings. 
\end{proof}

By looking at the generic fibre, one can check that  $\pi_2$  sends   $0_{\fc}$ (resp.\  $\infty_{\fc}$)  to $\infty(\fc\fp,\fp^{-1})$ (resp.\ $\infty(\fc\fp)$). The description of $\pi_1$ in Proposition \ref{formal-cusps}, along with an analogous analysis of $\pi_2$ gives the following result.

\begin{cor}\label{formal-completion}
\begin{enumerate}[(a)]
\item The formal completion of $\pi_1$ along $\infty_\fc$ gives an isomorphism
$\cY_\infty(\fc) \to \cX_\infty(\fc)$, which corresponds to the identity on~$M_\infty(\fc;\Z_p)$.
\item The formal completion of $\pi_1$ along $0_\fc$ gives a morphism
$\cY_0(\fc) \to \cX_\infty(\fc)$, which corresponds to the natural inclusion
$M_\infty(\fc;\Z_p) \hookrightarrow M_\infty(\fc \fp^{-1};\Z_p)$.
\item The formal completion of $\pi_2$ along $\infty_\fc$ gives a morphism
$\cY_\infty(\fc) \to \cX_\infty(\fc\fp)$, which corresponds to the natural inclusion
$M_\infty(\fc\fp;\Z_p) \hookrightarrow M_\infty(\fc;\Z_p)$.
\item The formal completion of $\pi_2$ along $0_\fc$ gives an isomorphism
$\cY_0(\fc) \to \cX_\infty(\fc \fp,\fp^{-1})$, which corresponds to the identity on~$M_\infty(\fc\fp^{-1};\Z_p)$.
\end{enumerate}
\end{cor}

Note that in view of  \cite[IV.7.8.3(vii)]{EGA}, Proposition \ref{formal-cusps} implies that 
$\overline{Y(\fc)}_{\F_p}$ has (at most) two horizontal (in the sense of \cite{goren-kassaei}) geometrically irreducible  components,  one containing  $0_{\fc,\F_p}$ and the other  containing $\infty_{\fc,\F_p}$.

\subsubsection{Integrality of the normalised trace}
 Let $U\to \Spec(\Z_p)$ be any open affine  in $\overline{X(\fc)}$ containing $\infty(\fc)$. 
  By \S\ref{q-exp} the  local  completed ring of  $U$   at  $\infty(\fc) $  is given by $M_\infty(\fc;\Z_p)$. 

The sheaves $\pi_i^*\overline{\underline{\omega}}^{\otimes k}$ ($i=1,2$) can be trivialised
over $\mathcal{Y}_\infty(\fc)\coprod \mathcal{Y}_0(\fc)$  yielding the following  commutative diagram: 
\begin{equation}\label{defn-Tp}
\xymatrix{ \rH^0(\pi_1^{-1}(U),\pi_2^*\overline{\underline{\omega}}^{\otimes k}) \ar[rr] \ar[d]& & 
M_\infty(\fc;\Z_p)\ar@{^{(}->}[d]\ar@{}[r]|-{\displaystyle\times} & M_\infty(\fc\fp^{-1};\Z_p) \ar@{^{(}->}[d]\\
\rH^0(\pi_1^{-1}(U_{\Q_p}),\pi_2^*\overline{\underline{\omega}}^{\otimes k}) \ar@{^{(}->}[rr] \ar[d]& &  M_\infty(\fc;\Q_p)\ar[d]_{ \Nm(\fp)^k}\ar@{}[r]|-{\displaystyle\times} & M_\infty(\fc\fp^{-1};\Q_p) \ar@{=}[d]\\
\rH^0(\pi_1^{-1}(U_{\Q_p}),\pi_1^*\overline{\underline{\omega}}^{\otimes k}) \ar@{^{(}->}[rr] \ar[d]_{\Nm(\fp)^{-1}\Tr(\pi_1)}& &  M_\infty(\fc;\Q_p)\ar[d]_{ \Nm(\fp)^{-1}}^{\displaystyle  \quad+}\ar@{}[r]|-{\displaystyle\times} & M_\infty(\fc\fp^{-1};\Q_p) \ar[dl]^{t_\fp}\\
\rH^0(U_{\Q_p},\overline{\underline{\omega}}^{\otimes k}) \ar@{^{(}->}[rr]& &
 M_\infty(\fc;\Q_p) & },
\end{equation}
where all horizontal morphisms are $q$-expansions at $\pi_1^{-1}(\infty(\fc))=\{\infty_\fc,0_\fc\}$ (see Corollary \ref{formal-completion}), except the last one which is the usual $q$-expansions map \eqref{eq:q-exp} at $\infty(\fc)$.
The middle vertical arrow comes from the connecting morphism of sheaves $\pi_2^*\overline{\underline{\omega}}^{\otimes k}\to \pi_1^*\overline{\underline{\omega}}^{\otimes k}$ described in \cite[\S1.11.1]{kisin-lai} and 
$t_\fp: M_\infty(\fc\fp^{-1};\Q_p)\to M_\infty(\fc;\Q_p)$ denotes the normalised trace (see~\S\ref{adelic-q-exp}). 
The commutativity of the middle and lower square follow from well known computations over  $\Q_p$.
The commutativity of the upper square comes from the commutativity of $q$-expansion and base change. Note that we do {\it not} claim the injectivity of the upper horizontal map, since when $\pi_1$ is not finite $\pi_1^{-1}(U_{\F_p})$  has  geometrically irreducible components  which contain neither  $\infty_{\fc,\F_p}$ nor $0_{\fc,\F_p}$.  

Denote by $\eta$ the composite of the morphisms of the first column in \eqref{defn-Tp}. Following the 
vertical maps on the right side of \eqref{defn-Tp} implies that if $k\geq 1$ then
$$\eta(\rH^0(U,\pi_{1*} \pi_2^*\overline{\underline{\omega}}^{\otimes k}))\subset 
\rH^0(U_{\Q_p},\overline{\underline{\omega}}^{\otimes k})\cap M_\infty(\fc;\Z_p)$$

Proposition  \ref{prop:q-exp}  implies that 
$$\rH^0(U_{\Q_p},\overline{\underline{\omega}}^{\otimes k})\cap M_\infty(\fc;\Z_p)
=\rH^0(U,\overline{\underline{\omega}}^{\otimes k}),$$
and that for any $U'\subset U$ as above the corresponding $\eta'$ and $\eta$ agree. 

Consider a  finite open affine cover $(U_i)_{i\in I(\fc)} $ of $\overline{X(\fc)}$ such that $\infty(\fc) \in U_i$ for all $i\in I(\fc)$ (one may, for example, consider an embedding of  $\overline{X(\fc)}$ in a projective space over $\Z_p$ and take complements of hyperplane sections which do not meet  $\infty(\fc) $). 
Let $I=\bigcup_{\fc\in \cC\ell_F^+} I(\fc)$ with $\fc$ running over the fixed  set of representatives of~$\cC\ell_F^+$. Then  $(U_i)_{i\in I} $ is an open affine cover  of $\overline{X}$. By the above discussion we obtain: 

\begin{prop} \label{Tp-integral}
For  $k\geq 1$, diagram \eqref{defn-Tp}  defines a   morphism of sheaves  $\eta: \pi_{1*} \pi_2^*\overline{\underline{\omega}}^{\otimes k}\to \overline{\underline{\omega}}^{\otimes k}$ on  $\overline{X}$. 
\end{prop}

\subsubsection{Definition of  $T_{\fp}^{(k)}$}\label{sec:3.3.4}

The adjunction morphism   $\overline{\underline{\omega}}^{\otimes k}_{\F_p}\to 
\pi_{2*} \pi_2^* \overline{\underline{\omega}}^{\otimes k}_{\F_p}$ of sheaves on $\overline{X}_{\F_p}$ induces: 
\begin{equation}\label{adj}
 \rH^0(\overline{X}_{\F_p},\overline{\underline{\omega}}^{\otimes k}_{\F_p})\overset{ \pi_2^*}{\longrightarrow}  \rH^0(\overline{X}_{\F_p},\pi_{2*} \pi_2^*\overline{\underline{\omega}}^{\otimes k}_{\F_p})
=\rH^0(\overline{Y}_{\F_p},\pi_2^*\overline{\underline{\omega}}^{\otimes k}_{\F_p})=
\rH^0(\overline{X}_{\F_p},\pi_{1*} \pi_2^*\overline{\underline{\omega}}^{\otimes k}_{\F_p}).
\end{equation}

By Proposition \ref{Tp-integral} one has a morphism: 
\begin{equation}\label{eta-p}
\rH^0(\overline{X}_{\F_p},(\pi_{1*} \pi_2^*\overline{\underline{\omega}}^{\otimes k})\otimes\F_p)
\overset{\eta\otimes\F_p}{\longrightarrow} \rH^0(\overline{X}_{\F_p},\overline{\underline{\omega}}^{\otimes k}_{\F_p} ),
\end{equation}
which we will now extend to $\rH^0(\overline{X}_{\F_p},\pi_{1*} \pi_2^*\overline{\underline{\omega}}^{\otimes k}_{\F_p})$. The short exact sequence 
$$0\to   \pi_2^*\overline{\underline{\omega}}^{\otimes k} \overset{\cdot p}{\to}  \pi_2^*\overline{\underline{\omega}}^{\otimes k} \to
 \pi_2^*\overline{\underline{\omega}}^{\otimes k}_{\F_p} \to 0$$
of sheaves on $\overline{X}$  yields the long exact sequence: 
\begin{equation*}
\begin{split}
0\to   \pi_{1*}\pi_2^*\overline{\underline{\omega}}^{\otimes k} \overset{\cdot p}{\to}  \pi_{1*}\pi_2^*\overline{\underline{\omega}}^{\otimes k} \to & \pi_{1*}\pi_2^*\overline{\underline{\omega}}^{\otimes k}_{\F_p} \to \\
\to & \mathrm{R}^1\pi_{1*}\pi_2^*\overline{\underline{\omega}}^{\otimes k} \overset{\cdot p}{\to}   \mathrm{R}^1\pi_{1*}\pi_2^*\overline{\underline{\omega}}^{\otimes k} \to \mathrm{R}^1\pi_{1*}\pi_2^*\overline{\underline{\omega}}^{\otimes k}_{\F_p}
\end{split}
\end{equation*}

By \cite[Proposition 3.19]{ERX},
the support of $\mathrm{R}^1\pi_{1*}\pi_2^*\overline{\underline{\omega}}^{\otimes k}_{\F_p}$ has codimension at least $2$ in 
$\overline{X}_{\F_p}^\Ra$, {\it i.e.,}  its localisation $(\mathrm{R}^1\pi_{1*}\pi_2^*\overline{\underline{\omega}}^{\otimes k}_{\F_p})_x$ at any codimension $1$ point $x$ in  $\overline{X}_{\F_p}^\Ra$ vanishes.
Note that in {\it loc.\ cit.} $X$ is replaced by the so-called splitting model, which is different if $p$ divides $\Nm(\fd)$,
but it contains $X^\Ra$ as an open subscheme.
Localising at any such $x$ the above long exact sequence then yields 
$$(\mathrm{R}^1\pi_{1*}\pi_2^*\overline{\underline{\omega}}^{\otimes k})_x\otimes \F_p=
((\mathrm{R}^1\pi_{1*}\pi_2^*\overline{\underline{\omega}}^{\otimes k})\otimes \F_p)_x=0. $$
Since $\pi_1$ is proper, the sheaf   $\mathrm{R}^1\pi_{1*}\pi_2^*\overline{\underline{\omega}}^{\otimes k}$ is coherent  and Nakayama's lemma implies that 
$(\mathrm{R}^1\pi_{1*}\pi_2^*\overline{\underline{\omega}}^{\otimes k})_x=0$. 
The above long exact sequence then implies that the morphism of sheaves 
\begin{equation}\label{eta-s}
(\pi_{1*} \pi_2^*\overline{\underline{\omega}}^{\otimes k})\otimes\F_p \to 
\pi_{1*} \pi_2^*\overline{\underline{\omega}}^{\otimes k}_{\F_p}
\end{equation}
is an isomorphism outside a subscheme $\overline{X}_{\F_p}^\Ra$ of codimension at least~$2$.
Since the complement of $\overline{X}^\Ra$ in~$\overline{X}$ has codimension at least~$2$, \eqref{eta-s} is an isomorphism
outside a subscheme $Z$ of $\overline{X}_{\F_p}$ of codimension at least~$2$.

Proposition \ref{Tp-integral} then yields: 
\begin{equation}\label{eta}
\begin{split}
\rH^0(\overline{X}_{\F_p}, \pi_{1*} \pi_2^*\overline{\underline{\omega}}^{\otimes k}_{\F_p}) & \overset{\mathrm{res}}{\longrightarrow}  
\rH^0(\overline{X}_{\F_p}\backslash Z , \pi_{1*} \pi_2^*\overline{\underline{\omega}}^{\otimes k}_{\F_p})
 \underset{\sim}{\overset{\eqref{eta-s}}{\longleftarrow}}
  \rH^0(\overline{X}_{\F_p}\backslash Z , (\pi_{1*} \pi_2^*\overline{\underline{\omega}}^{\otimes k})\otimes\F_p)\to 
   \\
&  \overset{\eta\otimes\F_p}{\longrightarrow} \rH^0(\overline{X}_{\F_p}\backslash Z,\overline{\underline{\omega}}^{\otimes k}_{\F_p}) =
\rH^0(\overline{X}_{\F_p},\overline{\underline{\omega}}^{\otimes k}_{\F_p}),
\end{split}
\end{equation}
where the last equality follows from an algebraic version of Hartogs' theorem, since $\overline{X}_{\F_p}$ is Noetherian normal,
$\overline{\underline{\omega}}^{\otimes k}_{\F_p}$ is locally free and $Z$ is of codimension at least~$2$.

For $k\geq 1$, the   endomorphism  $T_{\fp}^{(k)\vee}$ of  $\rH^0(\overline{X}_{\F_p},\overline{\underline{\omega}}^{\otimes k}_{\F_p})$ is defined as  the composite of \eqref{adj} with \eqref{eta}. Finally  we set:
\begin{equation}
T_{\fp}^{(k)}=T_{\fp}^{(k)\vee}\circ \langle\fp\rangle: M_k(\fn;\F_p)\to   M_k(\fn;\F_p).
\end{equation}

\subsubsection{The effect of $T_\fp$ on $q$-expansions}

We first compute the effect of \eqref{adj} on $q$-expansions. 

By   Lemma \ref{formal-diamond} and Corollary \ref{formal-completion}   for any $f\in M_k(\fn;\F_p)$ the $q$-expansion  of $ \pi_2^*(f)\in \rH^0(\overline{Y}_{\F_p},\pi_2^*\overline{\underline{\omega}}^{\otimes k})$  at $\pi_1^{-1}(\infty(\fc))=\{\infty_\fc,0_\fc\}$ equals 
\begin{equation}\label{qexp-adj}
(\iota_\fp(f_{\fc\fp}), (\langle\fp^{-1}\rangle f)_{\fc\fp^{-1}})\in 
M_\infty(\fc;\F_p)\times M_\infty(\fc\fp^{-1};\F_p),  
\end{equation}
where $\iota_\fp: M_\infty(\fc\fp;\F_p)\to M_\infty(\fc;\F_p)$ is the natural inclusion (see~\S\ref{adelic-q-exp}).

Combining  \eqref{defn-Tp} and \eqref{qexp-adj} shows that for $f\in M_k(\fn;\F_p)$ one has:
$$ (T_{\fp}^{(k)\vee}f)_\fc= \Nm(\fp)^{k-1} f_{\fc\fp} + t_\fp((\langle\fp^{-1}\rangle f)_{\fc\fp^{-1}}),\text{ and}$$
\begin{equation}\label{qexp-Tp}
 (T_{\fp}^{(k)}f)_\fc=   t_\fp(f_{\fc\fp^{-1}})+ \Nm(\fp)^{k-1} (\langle\fp\rangle f)_{\fc\fp}, 
\end{equation}
where $t_\fp: M_\infty(\fc\fp^{-1};\F_p)\to M_\infty(\fc;\F_p)$ denotes the normalised trace (see~\S\ref{adelic-q-exp}).

Letting  $\fc$ run over the fixed   set of representatives of $\cC\ell_F^+$,  the facts established in \S\ref{adelic-q-exp} lead to
a proof of Theorem~\ref{thm:Tp-q}.

Note that by the $q$-expansion principle the  Hecke and diamond operators generate a commutative algebra. 

Due to its importance for the sequel, we draw the reader's attention to the fact that
an inspection of the action on formal $q$-expansions shows that, in characteristic~$p$,
$T_\fp^{(k)}$ acts like a $U_\fp$-operator as soon as $k>1$.

\subsection{The Hasse invariant}
In the sequel, the Hasse invariant $h \in M_{p-1}(\fn, 1;\Fbar_p)$ will play a fundamental
role for passing from weight~$1$ to weight~$p$; in combination with the Hecke operators $T_\fp$
for $\fp$ dividing~$p$, the Frobenius operators will be derived from it in the next subsection.
The Hasse invariant $h$ has $q$-expansion equal to~$1$ at the cusp at $\infty$ of each connected component of  $\overline{X_1(\fn)}$.

For the existence of the Hasse invariant, the reader can either refer to \cite[\S7.12-14]{andreatta-goren}
or \cite[\S 1.5]{kisin-lai}. By p.~740 of the latter reference, the Hasse invariant, which is {\em a priori} constructed
on $\overline{X^1_1(\fn)}$, is independent of the polarisation and thus $E$-invariant,  so that it descends to $\overline{X_1(\fn)}$,
as needed.

\subsection{The Frobenius operators}

In this section we define operators $V_P$ taking  weight $1$ forms to weight~$p$ forms. 
A main feature of our treatment is that its only ingredients are the Hecke operators $T_\fp^{(1)}$ and $T_\fp^{(p)}$
for primes $\fp \mid p$ and the total Hasse invariant.

Let $\varepsilon$  be a fixed $\Fbar_p^\times$-valued Hecke character of~$F$ of conductor dividing $\fn$.

For squarefree ideals $P \subset \fo$ dividing~$p$, we define the {\em Frobenius operator} 
\begin{equation} V_P: M_1(\fn,\varepsilon;\Fbar_p) \to M_p(\fn,\varepsilon;\Fbar_p)
 \end{equation}
 inductively as follows:
\begin{equation} V_1 := V_{\fo} := h \textnormal{ and }
V_{P\fp} := \varepsilon(\fp)^{-1} \left( V_P T_\fp^{(1)} - T_\fp^{(p)} V_P \right),  
\end{equation}
where $\fp $ is a prime ideal dividing~$(p)$ coprime to~$P$.

The description of the action on $q$-expansions in the following proposition shows that the $V_P$ 
do not depend on the order in which the prime divisors of~$P$ are used in the recursive definition,
hence they are well-defined.  It also shows that $V_P$ commutes with the Hecke operators $T_\fq$ for all $\fq\nmid p\fn$.

\begin{prop}
For every squarefree ideal $P \subset \fo$ dividing~$p$ one has: 
 \begin{equation}\label{eq:Vp}
 \begin{split}
 a((0), V_P f) & = a((0),f)[P^{-1}],   \text{ and}\\
a(\fr, V_P f) & = a(\fr/P,f) \text{ for all } (0)\neq\fr\subset \fo.
 \end{split}
 \end{equation}
\end{prop}

\begin{proof}
If $P=\fo$, the result is trivial.
Assume that the result holds for~$P$ and let $\fp$ be a prime ideal dividing~$(p)$ coprime to~$P$. Using 
\eqref{eq:Tp-general}, for all $(0)\neq\fr\subset \fo$  we have
\begin{equation*}
\begin{split}
a(\fr, V_{P\fp} f)
&= \varepsilon(\fp)^{-1} \left( a(\fr, V_P T_\fp^{(1)} f)- a(\fr, T_\fp^{(p)} V_P f) \right)\\
&= \varepsilon(\fp)^{-1} \left( a(\fr/P, T_\fp^{(1)} f)- a(\fr\fp, V_P f) \right)\\
&= \varepsilon(\fp)^{-1} \left( a(\fr\fp/P,f) +  a(\fr/(P\fp),\langle\fp\rangle  f) - a(\fr\fp/P, f) \right)\\
&= a(\fr/(P\fp),f), \text{while } 
\end{split}
\end{equation*}
\begin{equation*}
\begin{split}
a((0), V_{P\fp} f)
&= \varepsilon(\fp)^{-1} \left( a((0), V_P T_\fp^{(1)} f)- a((0), T_\fp^{(p)} V_P f) \right)\\
&= \varepsilon(\fp)^{-1} \left( a((0), T_\fp^{(1)} f)[P^{-1}] - a((0), V_P f)[\fp] \right)\\
&= \varepsilon(\fp)^{-1} \left( a((0),f)[\fp P^{-1}] +  a((0),\langle\fp\rangle  f)[(P\fp)^{-1}] - a((0), f)[\fp P^{-1}] \right)\\
&= a((0),f)[(P\fp)^{-1}].
\end{split}
\end{equation*}
Note that the calculations make sense even if $P$ does not divide~$\fr$; it is here that the squarefreeness is used.
\end{proof}

\begin{lemma}\label{lem:UV}
For $\fp \subset \fo$ a prime and $P \subset \fo$  squarefree, both dividing~$p$, we have
$$T_\fp^{(p)} V_P = \begin{cases}
V_{P/\fp}                       & \textnormal{ if } \fp\mid P,\\
V_P T_\fp^{(1)} - \varepsilon(\fp) V_{P\fp}  & \textnormal{ if } \fp\nmid P.
\end{cases}$$
\end{lemma}

\begin{proof}
We first compute the adelic $q$-expansion  of $T_\fp^{(p)} V_P f$:
\begin{equation*}
\begin{split}
 a(\fr, T_\fp^{(p)} V_P f) &= a(\fr \fp, V_P f) = a(\fr \fp/P, f), \text{and}\\
  a((0), T_\fp^{(p)} V_P f)& = a((0), V_P f)[\fp] = a((0), f)[\fp P^{-1}] .
 \end{split}
\end{equation*}
If $\fp \mid P$ this simplifies and equals $a(\fr, V_{P/\fp} f)$, as claimed.  
If $\fp\nmid P$ then 
\begin{equation*}
\begin{split}
  a(\fr, V_P T_\fp^{(1)} f)
&= a(\fr/P, T_\fp^{(1)} f)= a(\fr \fp / P, f) + \varepsilon(\fp) a(\fr/(P \fp), f)\\
&= a(\fr, T_\fp^{(p)} V_P f) + \varepsilon(\fp) a(\fr, V_{P\fp} f),
 \end{split}
\end{equation*}
\begin{equation*}
\begin{split}
  a((0), V_P T_\fp^{(1)} f)
&= a((0), T_\fp^{(1)} f)[P^{-1}]= a((0), f)[\fp P^{-1}] + \varepsilon(\fp) a((0), f)[(\fp P)^{-1}] \\
&= a((0), T_\fp^{(p)} V_P f) + \varepsilon(\fp) a((0), V_{P\fp} f),
 \end{split}
\end{equation*}
proving the claim.
\end{proof}

\section{Galois representations}

This section contains the proof of the main theorem.
It is based on the Hecke operators $T_\fp^{(k)}$ for $\fp \mid p$ and the Frobenius operators~$V_P$, constructed in the previous section.  A form will be called constant if its adelic $q$-expansion (see \S\ref{adelic-q-exp}) is reduced to its constant term  in  $\Fbar_p [\cC\ell_F^+] $. 
We start with a study of the $T_\fp^{(p)}$-action on  $h\cdot f$ where $f$ is a given non-constant  $T_\fp^{(1)}$-eigenform. 
Next we recall the construction of the Galois representation attached to a weight $1$ eigenform, by 
finding an ordinary eigenform of higher weight in characteristic~$0$ with congruent eigenvalues. 
Wiles' theorem  describing the local behaviour at $\fp$ of ordinary Galois representations along with some 
commutative algebra allows us to prove the theorem. 

Let  $\Sigma$ be any finite set of primes of~$F$ containing those dividing $\fn$, and  
consider the abstract Hecke algebra
$$\T' = \Zbar_p[T_\fq \;|\; \fq \subset \fo \textnormal{ prime }, \fq \not\in \Sigma,  \fq \nmid p].$$

Let $\varepsilon$ be a fixed $\Fbar_p^\times$-valued Hecke character of~$F$ of conductor dividing $\fn$.

\subsection{Hecke orbits}

Note that whereas all the forms $V_P f$, for a given non-constant  $\T'$-eigenform $f$ of weight $1$, 
are common eigenvectors for $\T'$ sharing the same  eigenvalues,
the following proposition will show that $T_\fp^{(p)}$ will never act as a scalar on their $\Fbar_p$-linear span,
although $\fp$ does not divide the level~$\fn$. 
This phenomenon of failure of strong multiplicity one in characteristic~$p$ has been studied in detail 
by one of the authors in \cite{wiese} when $F=\Q$ and is called ``doubling''.
As we will see, for a general $F$, one has ``doubling'' at each prime dividing~$p$.

\begin{prop}\label{prop:heckeorbit}
Let $f \in M_1(\fn,\varepsilon;\Fbar_p)$ be a non-constant eigenform for $T_\fp^{(1)}$ with eigenvalue $\lambda_\fp$ for all $\fp \mid p$.

\begin{enumerate}[(a)]
\item\label{item:ho-a} The elements $V_P f$, with $P$ running through all squarefree ideals of~$\fo$ dividing~$p$,  are linearly independent over $\Fbar_p$. Denote by $W$ their $\Fbar_p$-linear span.

\item\label{item:ho-b} For all primes $\fp$ dividing  $p$, the operator $T_\fp^{(p)}$ preserves  $W$ and is 
annihilated by  $X^2 - \lambda_\fp X + \varepsilon(\fp)$. 
In particular $T_\fp^{(p)}$ is invertible on~$W$. 

\item\label{item:ho-c} Fix a prime $\fp$ dividing~$p$ and, for all primes $\fp' \mid p $ distinct from~$\fp$,
fix a root $\alpha_{\fp'}$ of $X^2 - \lambda_{\fp'} X + \varepsilon(\fp')$.
Define
$$W_{\fp}=W[T_{\fp'}^{(p)} - \alpha_{\fp'} \;,\; \fp \neq \fp'\mid p]$$  
as the $\Fbar_p$-subspace of~$W$ on which $T_{\fp'}^{(p)}$ acts by scalar multiplication by $\alpha_{\fp'}$
for all primes $\fp' \mid p $ distinct from~$\fp$.  
Then $W_{\fp}$ has dimension $2$ and $T_\fp^{(p)}$ acts on it with minimal polynomial $X^2 - \lambda_\fp X + \varepsilon(\fp)$.
In particular, if the polynomial $X^2 - \lambda_\fp X + \varepsilon(\fp)$ has a double root, then 
$T_\fp^{(p)}$ does not act semi-simply on $W_{\fp}$. 
\end{enumerate}
\end{prop}

\begin{proof}
\eqref{item:ho-a}
Since $f$ is non-constant, there exists a non-zero integral ideal $\fr$ such that $a(\fr,f) \neq 0$,
but $a(\fr',f)=0$ for all proper divisors $\fr'$ of $\fr$.
We first show that $\fr$ is relatively prime to~$(p)$. Indeed, if $\fp\mid p$ were a prime dividing~$\fr$,
then~\eqref{eq:Tp-general} would imply that
$$\lambda_\fp\cdot  a(\fr\fp^{-1}, f) 
=a(\fr\fp^{-1}, T_\fp^{(1)} f)  = a(\fr ,f) + \varepsilon(\fp) a(\fr\fp^{-2}, f),$$
which is impossible since $a(\fr\fp^{-2}, f)=a(\fr\fp^{-1}, f)=0$ and $a(\fr ,f)\neq 0$.
In fact, this argument shows that $\fr$ is not divisible by any prime~$\fq$
such that  $f$ is a $T_\fq^{(1)}$-eigenform.
One should hence expect to be able to take $\fr=\fo$, {\it i.e.}, to take $f$ normalised, but we will not need this.

Suppose now that we have a linear combination $0 = \sum_{P \mid p} b_P V_Pf$.
We show $b_P=0$ by induction on the number of primes dividing~$P$.
To start the induction, note that $b_\fo = 0$ as $a(\fr, V_P f)=a(\fr/P, f)=0$ if $P \neq \fo$, and $a(\fr, V_\fo f)=a(\fr, f)\neq 0$.
Let now $R \mid p$ be squarefree and suppose $b_Q = 0$ for all proper divisors $Q \mid R$. Then using  \eqref{eq:Vp}
and the fact that $\fr$ is relatively prime to $P$ we find
$$0 = \sum_{P \mid p} b_P a(\fr R,V_Pf) =  \sum_{P \mid p} b_P a(\fr R/P,f)=
\sum_{P \mid R} b_P a(\fr R/P,f) = a(\fr, f) b_R,$$
hence $b_R=0$  as claimed.

\eqref{item:ho-b}
By Lemma~\ref{lem:UV} the space $W$  is stable under $T_{\fp}^{(p)}$ for all primes $\fp$ dividing $p$.
 For  $\fp\mid P$ we have 
$$T_\fp^{(p)2} V_P f = T_\fp^{(p)} V_{P/\fp}f = \lambda_\fp V_{P/\fp}f - \varepsilon(\fp) V_P f\\
= (\lambda_\fp T_\fp^{(p)} - \varepsilon(\fp)) V_P f,$$
whereas for $\fp\nmid P$ 
$$ T_\fp^{(p)2} V_P f = T_\fp^{(p)} (\lambda_\fp V_P - \varepsilon(\fp) V_{P\fp})f = (\lambda_\fp T_\fp^{(p)} - \varepsilon(\fp)) V_P f,$$
showing that $T_\fp^{(p)2} - \lambda_\fp T_\fp^{(p)} + \varepsilon(\fp)\id $ annihilates~$W$.

\eqref{item:ho-c} For any squarefree ideal $P$ dividing~$p$, we let $Z_P$ be the $\Fbar_p$-subspace of~$W$ having basis $V_Q f$ with $Q$ running through all divisors of~$P$.
Since $f$ is an eigenform for $T_\fp^{(1)}$, Lemma~\ref{lem:UV} implies that  $Z_{P \fp}$ is preserved by 
$T_\fp^{(p)} $ for all primes $\fp \nmid P$ and
\begin{equation}\label{eq:doubling}
Z_{P \fp} = Z_ P \oplus T_\fp^{(p)} Z_P,
\end{equation}
the sum being direct since $\dim(Z_{P \fp})=2\dim(Z_P)$ by~\eqref{item:ho-a}. 

Now, we show that $Z_P[T_{\fp'}^{(p)} - \alpha_{\fp'} \;,\; \fp'\mid P]$
is $1$-dimensional  by induction on the number of primes dividing~$P$.
Note that $Z_\fo$ is a line spanned by $V_\fo f$. 
Suppose that $Z_P[T_{\fp'}^{(p)} - \alpha_{\fp'} \;,\; \fp'\mid P]=\Fbar_p g$ and let $\fp\mid p$ be a prime not dividing~$P$.
By \eqref{eq:doubling} one has 
$$Z_{P \fp} [T_{\fp'}^{(p)} - \alpha_{\fp'} \;,\; \fp'\mid P]= \Fbar_p g\oplus \Fbar_p T_\fp^{(p)} g,$$
on which  $T_\fp^{(p)}$ does not act as a scalar, hence by \eqref{item:ho-b}
its minimal polynomial is $X^2 - \lambda_\fp X + \varepsilon(\fp)$.
It follows that  $Z_{P \fp} [T_{\fp'}^{(p)} - \alpha_{\fp'} \;,\; \fp'\mid P \fp]$ is $1$-dimensional, too. 

Taking $P$ to be the product of all primes dividing~$p$ and distinct from~$\fp$,
one gets $W_\fp= \Fbar_p g\oplus \Fbar_p T_\fp^{(p)} g$, yielding the desired result. 
\end{proof}

By counting dimensions, the above proposition implies that one can characterise $W$ as the smallest subspace of $M_p(\fn,\varepsilon;\Fbar_p)$ containing $h\cdot f$ that is stable under $T_\fp^{(p)}$ for all $\fp\mid p$.

\subsection{Construction of the Galois representations}\label{ssc:cons}

Let $f \in M_1(\fn;\Fbar_p)$ be any non-zero eigenform for~$\T'$.
By Lemma \ref{lem:lift}, there exists an integer $k_0>1$ such that in weight  $k=1+k_0(p-1)$ one has an isomorphism: 
$$M_k(\fn ; \Zbar_p) \otimes \Fbar_p \simeq M_k(\fn ; \Fbar_p).$$
Since $h^{k_0}f\in M_k(\fn ; \Fbar_p)$ is a $\T'$-eigenform with the same eigenvalues as $f$, by 
Deligne-Serre \cite[Lemma 6.11]{deligne-serre} there exists a $\T'$-eigenform $g \in M_k(\fn ; \Zbar_p)$
the eigenvalues of which lift those of~$f$.
We define then  $\rho_f:G_F\to \GL_2(\Fbar_p)$ 
as the semi-simplification of the reduction modulo $p$ of the $p$-adic Galois representation 
$\rho_{g}$ with representation space $V(g)$ attached to~$g$.
If $g$ is a cusp form, the Galois representation attached to~$g$ exists by~\cite{taylor}
 (since $g$ is ordinary, as we show below, the existence already follows from~\cite{wiles}).
The general case follows from the direct sum decomposition of $M_k(\fn,\C)$ into the cuspidal subspace and the  Eisenstein subspace.
The latter has a basis consisting of Eisenstein series attached to pairs of Hecke  characters; these Eisenstein series
are Hecke eigenforms and the direct sum of the two characters is the Galois representation attached to them (see \cite[\S 1.5]{wiles-padic}). 

By construction, $\rho_f$ is uniquely determined by $f$
since it is semi-simple, unramified at all $\fq\notin\Sigma$  not dividing~$p$,
and $\Tr(\rho_f(\Frob_\fq))$ equals the $T_\fq^{(1)}$-eigenvalue of  $f$.

\subsection{Ordinarity}
Let $f \in M_1(\fn;\Fbar_p)$ be any non-constant  $\T'$-eigenform.  
Since the operators $\langle\fq\rangle$ for $\fq\notin\Sigma$ commute with the action of  $\T'$, there exists an $\Fbar_p^\times$-valued Hecke character $\varepsilon$ of $F$ of conductor dividing $\fn$ and a form in  $M_1(\fn, \varepsilon ;\Fbar_p)$  sharing the same $\T'$-eigenvalues as $f$, that we will still denote $f$. Furthermore, since $\T'[T_{\fp}^{(1)}, \fp\mid p]$ acts  commutatively  on $M_1(\fn, \varepsilon ;\Fbar_p)$, we may and do assume that $f$ is also  an eigenform for  $T_{\fp}^{(1)}$ with some eigenvalue $\lambda_\fp$ for all primes $\fp$ dividing~$p$.

For every prime  $\fp$ dividing $p$, fix a root $\alpha_\fp$ of $X^2 - \lambda_\fp X + \varepsilon(\fp)$, which is never $0$.  By Proposition \ref{prop:heckeorbit}\eqref{item:ho-c}, there exists a $\T'$-eigenform $f_\alpha \in M_p(\fn,\varepsilon;\Fbar_p)$ sharing with $f$ the same $\T'$-eigenvalues  such that in addition $T_{\fp}^{(p)}  f_\alpha  = \alpha_\fp  f_\alpha $, for all $\fp\mid p$.  

For $k$ as in \S\ref{ssc:cons}, there exists a $\T'[T_{\fp}^{(p)}, \fp\mid p]$-eigenform $g \in M_k(\fn ; \Zbar_p)$ the eigenvalues of which lift those of $f_\alpha$, in particular $g$ is $\fp$-ordinary at all primes $\fp$ dividing $p$ and $T_{\fp}^{(k)}g=\tilde \alpha_\fp g$
with $\tilde \alpha_\fp$ lifting~$\alpha_\fp$.
Since the operators $\langle\fq\rangle$ for $\fq\notin\Sigma$ commute with the action of  $\T'[T_{\fp}^{(p)}, \fp\mid p]$ on 
$\in M_k(\fn ; \Zbar_p)$, we can assume in addition that $g$ has central character $\tilde \varepsilon$ lifting~$\varepsilon$.

By a theorem of Wiles \cite[Theorem~2]{wiles} (applied to a $p$-stabilisation of~$g$) for each $\fp\mid p$ one has a short exact sequence of $\Qbar_p[D_{\fp}]$-modules: 
 \begin{equation}\label{eq:ord}
  0 \to  V(g)^+ \to   V(g) \to V(g)^- \to  0 
 \end{equation}
such that $V_{g}^-$ has dimension $1$ over $\Qbar_p$ and the decomposition group $D_{\fp}$ of~$\fp$ acts on it by the unramified character sending $\Frob_\fp$ to the unique unit root of 
$X^2- \tilde \alpha_\fp X +\tilde \varepsilon(\fp)\Nm(\fp)^{k-1}$. 
 
Reduction modulo $p$ yields the following result. 

\begin{prop} \label{prop:ord} For any non-constant  $f$ as above and  for every prime  $\fp$ dividing $p$  and 
for any root $\alpha_\fp$ of $X^2 - \lambda_\fp X + \varepsilon(\fp)$, the representation 
$\rho_{f|D_{\fp}}$ admits a $1$-dimensional unramified quotient on which $\Frob_\fp$  acts by $\alpha_\fp$. 
\end{prop}

\subsection{Proof of the main theorem}\label{main-proof}

Let $f \in M_1(\fn;\Fbar_p)$ be a $\T'$-eigenform as in the theorem and fix a prime $\fp$ dividing~$p$. 
As in the previous section we may and do assume that $f\in M_1(\fn, \varepsilon ;\Fbar_p)$ for some  $\Fbar_p^\times$-valued 
Hecke character $\varepsilon$ of $F$ of conductor dividing $\fn$ and that for all $\fp'$ dividing $p$ we have $T_{\fp'}^{(1)} f= \lambda_{\fp'} f$.
 For all $\fp'$ dividing~$p$, we choose a root $\alpha_{\fp'}$ of $X^2- \lambda_{\fp'} X + \varepsilon(\fp')$.

Assume first that $f$ is constant.  By~\eqref{eq:Tp-general}, for any prime $\fq\notin \Sigma$, $\fq\nmid p$ and any 
constant weight $1$ form $g$ we have
$$T_\fq^{(1)} g  = g[\fq] +\varepsilon(\fq) g[\fq^{-1}].$$
Now, consider the action by translation of the finite abelian group $\cC\ell_F^+$ on its 
group ring $\Fbar_p [\cC\ell_F^+]$. Since this action clearly commutes with the action of $\T'$, one can assume that 
$f$ is an eigenform for this $\cC\ell_F^+$-action, {\it i.e.} up to a non-zero scalar $f=\sum_{\fc\in \cC\ell_F^+} \varphi(\fc)[\fc^{-1}]$ for some 
character $\varphi:\cC\ell_F^+ \to \Fbar_p^\times$. It follows then from the above equation that the 
$T_\fq^{(1)}$-eigenvalue of $f$ equals $\varphi(\fq) + \varepsilon(\fq) \varphi^{-1}(\fq)$. 
This shows that the semisimple Galois representation $\rho_f$ is isomorphic to the Galois representation corresponding by
class field theory to $\varphi \oplus \varepsilon \varphi^{-1}$. In particular, it is unramified above~$p$.
Henceforth we will assume that $f$ is not constant. 

If $X^2- \lambda_\fp X + \varepsilon(\fp)$ has two distinct roots $\alpha_\fp$  and $\beta_\fp$  then Proposition \ref{prop:ord} implies that $\rho_{f|D_{\fp}}$ admits two distinct unramified quotients on which $\Frob_\fp$  acts by $\alpha_\fp$ and $\beta_\fp$, respectively, which proves that   $\rho_{f|D_{\fp}}$ is unramified in this case. Moreover, in this case, the trace of $\rho_f(\Frob_{\fp})$ equals $\alpha_\fp+\beta_\fp=\lambda_\fp$.

If $\rho_f$ is a sum of two characters, then Proposition \ref{prop:ord} implies that (at least) one of them is
unramified at~$\fp$. Moreover, as is well known, the determinant of $\rho_f$ equals $\varepsilon$, hence is unramified in $\fp$, 
which proves that $\rho_{f|D_{\fp}}$ is unramified also in this case.

In the sequel we treat the remaining case where $f$ is non-constant, $\rho_f$ is irreducible and $X^2- \lambda_\fp X + \varepsilon(\fp)$  has a double root $\alpha_\fp$.  Denote by $\fm'$ the maximal ideal of $\T'$ corresponding to~$f$.
Moreover, let $\fm$ be the maximal ideal of $\T'[T_{\fp'}^{(p)}, \fp'\mid p]$ corresponding to $f$ and the choice of the $\alpha_{\fp'}$'s.

Let $\T'_k$  be the image of the abstract Hecke algebra $\T'$  in the ring of 
endomorphisms of $M_k(\fn; \Zbar_p)_\fm$
(localisation of $M_k(\fn; \Zbar_p)$ at~$\fm$),
where $k=1+k_0(p-1)$ is as in Lemma~\ref{lem:lift}. 
In particular, the algebra generated by $\T'$ acting on $M_k(\fn; \Fbar_p)_\fm$ is naturally isomorphic to $\T'_k\otimes_{\Zbar_p} \Fbar_p$.

By the $q$-expansion principle, $M_k(\fn ; \Zbar_p)$ is a torsion free $\Zbar_p$-module,
hence $\T'_k$ is torsion free. Moreover,  $\T'_k$ is reduced as it is a $\Zbar_p$-lattice in the semi-simple algebra
$$ \T'_k\otimes_{\Zbar_p} \Qbar_p\simeq  \prod_{g \in \cN} \Qbar_p, $$ 
where $\cN$ denotes the set of all newforms contributing to  $M_k(\fn; \Zbar_p)_\fm \otimes \Qbar_p$. 

To make the rest of the argument more transparent, we prove a useful fact.   

\begin{lemma}\label{lem:key}
The Hecke operator $T_{\fp}^{(k)}$ acting on $M_k(\fn; \Fbar_p)_\fm$
does not belong to $\T'_k\otimes_{\Zbar_p} \Fbar_p$.
Moreover, $T_{\fp}^{(k)}$ acting on $M_k(\fn; \Zbar_p)_\fm$  belongs to 
$\T'_k\otimes_{\Zbar_p} \Qbar_p$, but does not belong to~$\T'_k$.

\end{lemma}
\begin{proof}
Since $\fp$ does not divide the level~$\fn$, Strong Multiplicity One applied to each 
$g\in \cN$ implies that $T_{\fp}^{(k)}\in \T'_k\otimes_{\Zbar_p} \Qbar_p$. 

By Lemma \ref{lem:lift} one has  a $\T'[T_{\fp}]$-equivariant  isomorphism:
$$M_k(\fn ; \Zbar_p)_\fm  \otimes \Fbar_p \simeq M_k(\fn ; \Fbar_p)_\fm. $$
It follows that if $T_{\fp}^{(k)}$ belonged to  $\T'_k$ acting on $M_k(\fn ; \Zbar_p)_\fm$, then
$T_{\fp}^{(k)}$ acting on $M_k(\fn ; \Fbar_p)_\fm$ would belong to $\T'_k \otimes \Fbar_p$ and  we will now show that this is impossible.

Define $W_\fp$ as in Proposition~\ref{prop:heckeorbit}\eqref{item:ho-c}. Then,
as one can see from $q$-expansions, there are $\T'[T_{\fp}]$-equivariant  inclusions  
$$ M_k(\fn ; \Fbar_p)_\fm \supset  h^{k_0-1}  M_p(\fn ; \Fbar_p)_\fm \supset h^{k_0-1}  M_p(\fn ; \Fbar_p)[\fm']_\fm \supset h^{k_0-1} W_{\fp}.$$
Hence, if $T_{\fp}^{(k)}$ belonged to $\T'_k \otimes \Fbar_p$, then $T_{\fp}^{(p)}$ would also belong to
$\T'$ acting on $W_{\fp}$. However $\T'$ acts on $W_{\fp}$ by a character, whereas by 
Proposition~\ref{prop:heckeorbit}\eqref{item:ho-c} the action of $T_{\fp}^{(p)}$ on $W_{\fp}$  is not semisimple. 
\end{proof}

Since  $\rho_f$ is irreducible, there exists by  \cite[Th\'eor\`eme~2]{carayol}  a free $\T'_k$-module $\cM$ of  rank $2$ with a continuous action of  $G_F$  such that one has a $G_F$-equivariant isomorphism of 
$\Qbar_p$-vector spaces:
$$\cM\otimes_{\Zbar_p} \Qbar_p\simeq \prod_{g\in \cN} V(g).$$
For  all $\fp'\mid p$,  the $T_{\fp'}^{(k)}$-eigenvalue of any  $ g\in \cN$ belongs to $\Zbar_p^\times$ (since it reduces  to $\alpha_{\fp'}\in  \Fbar_p^\times$),  hence there exists a short exact sequence
of $\Qbar_p[D_{\fp}]$-modules  (\ref{eq:ord}) for each $ g\in \cN$. 

Letting $\cM^+ = \cM \cap \prod_{g\in \cN} V(g)^+$ and $\cM^-$ be  the image of $\cM$ in $\bigoplus_{g\in \cN} V(g)^-$, one obtains a short exact sequence of $\T'[D_{\fp}]$-modules:
\begin{equation}\label{eq:M-exact}
0 \to  \cM^+ \to \cM \to \cM^- \to 0.
\end{equation}
Reducing (\ref{eq:M-exact}) modulo~$\fm'$, we get an exact sequence of $\Fbar_p[D_{\fp}]$-modules
$$ \cM^+/\fm' \cM^+ \to \cM/\fm' \cM \to \cM^-/\fm' \cM^- \to 0.$$
Since $\cM^-\neq 0$ Nakayama's lemma for the local algebra $\T'_k$ implies $\cM^-/\fm'\cM^-\neq 0$. 

If  $\dim_{\Fbar_p}(\cM^-/\fm'\cM^-)=1$, then   Nakayama's lemma yields a surjective 
homomorphism  $\T'_k \twoheadrightarrow \cM^-$ of $\T'_k$-modules. Since $\T'_k$ and $\cM^-$ have the same $\Zbar_p$-rank (equal to half the $\Zbar_p$-rank of $\cM$) we deduce that $\cM^-$ is free of rank $1$ over $\T'_k$, in particular $\Frob_\fp$ acts on $\cM^-$ as some element 
$U\in (\T'_k)^\times$. By (\ref{eq:ord}), the image of $U$ in $\T'_k \otimes \Fbar_p$ equals $T_\fp^{(k)}$ since for all $g \in \cN$ the
eigenvalue of $U$ on $g$ is the unique unit root of the Hecke polynomial $X^2-T_{\fp}^{(k)}X+ \langle \fp \rangle \Nm(\fp)^{k-1}$ of $g$ at~$\fp$.
Hence the endomorphism $T_{\fp}^{(k)}$ of $M_k(\fn; \Fbar_p)_\fm$ belongs to $\T'_k\otimes_{\Zbar_p} \Fbar_p$, contradicting Lemma~\ref{lem:key}. 

Therefore  $\dim_{\Fbar_p}(\cM^-/\fm'\cM^-)=2$ hence  $\cM/\fm' \cM \simeq \cM^-/\fm' \cM^-$ implying that 
$\rho_f$ is unramified at~$\fp$.
Moreover, since $\cM^-/\fm' \cM^-$ is a quotient of two lattices in $\prod_{g\in \cN} V(g)^-$ and since 
$\Frob_\fp$ acts on each $V(g)^-$ by a scalar reducing modulo $p$ to $\alpha_\fp$, if follows that 
$\Frob_\fp$ acts  on $\cM/\fm' \cM$ via a matrix 
$\left(\begin{smallmatrix} \alpha_\fp & * \\ 0 & \alpha_\fp \end{smallmatrix} \right)$, thus has trace~$\lambda_\fp=2\alpha_\fp$.

\begin{rem}\label{u-operator}
The fact that the unramified quotient $\cM^-$ of $\cM$ is not free over $\T'_k$, means that the representations 
$G_F\to \GL_2(\T'_k)$ is not ordinary at $\fp$, but would become ordinary once the $U$-operator has been adjoined to  $\T'_k$. 
\end{rem}

\begin{small}
\noindent
\begin{tabular}{lll}
Mladen Dimitrov & \hspace*{.6cm}& Gabor Wiese\\ 
Universit\'e Lille 1 && University of Luxembourg\\
UFR Mathématiques && Mathematics Research Unit\\
Cit\'e Scientifique && 6, avenue de la Fonte\\
59655 Villeneuve d'Ascq Cedex && L-4364 Esch-sur-Alzette\\
FRANCE && LUXEMBOURG\\ 
{\tt mladen.dimitrov@math.univ-lille1.fr} && {\tt gabor.wiese@uni.lu}\\
\end{tabular}
\end{small}

\end{document}